%% file: mainRCPParXiv.tex
\documentclass[journal]{IEEEtran}
\ifCLASSINFOpdf
\usepackage[pdftex]{graphicx}
\else
\fi
\usepackage{subfigure}

\usepackage{amsmath, amssymb, amsthm}
\usepackage{amsfonts}
\usepackage{algorithm}
\usepackage{algpseudocode}
\ifCLASSOPTIONcompsoc
\usepackage[caption=false,font=normalsize,labelfont=sf,textfont=sf]{subfig}
\else
\usepackage[caption=false,font=footnotesize]{subfig}
\fi

	\usepackage{color}
	\usepackage{xcolor}
	\usepackage{tabularx}
	\usepackage{multirow}
	\usepackage{tablefootnote}
	
	\usepackage{threeparttable}
	\newtheorem{assumption}{Assumption}
	\newtheorem{lemma}{Lemma}
	\newtheorem{remark}{Remark}
	\newtheorem{theorem}{Theorem}
	
	\input{notation}

\begin{document}
		%
		\title{A  Robust Compressed Push-Pull Method for Decentralized Nonconvex Optimization}
		%
		%
		%
		
		
		\author{Yiwei Liao, Zhuorui Li, Shi Pu, and Tsung-Hui Chang
			\thanks{Yiwei Liao is with the School of Electrical Engineering, Sichuan University, China. 
				Zhuorui Li is with the H. Milton Stewart School of Industrial and System Engineering, Georgia Institute of Technology, Atlanta, USA.  
				Shi Pu is with the School of Data Science, The Chinese University of Hong Kong, Shenzhen, China.
				Tsung-Hui Chang is with the School of Science and Engineering, The Chinese University of Hong Kong, Shenzhen, China.
				{\tt\small (emails: liaoyiwei@scu.edu.cn, lizhuorui27@gmail.com,  pushi@cuhk.edu.cn, changtsunghui@cuhk.edu.cn)}}%
		}

	\maketitle
	
	\begin{abstract}
		In the modern paradigm of multi-agent networks, communication has become one of the main bottlenecks for decentralized optimization, where a large number of agents are involved in minimizing the average of the local cost functions. In this paper, we propose a robust compressed push-pull algorithm (RCPP) that combines gradient tracking with communication compression. In particular, RCPP is  robust under a much more general class of compression operators that allow both relative and absolute compression errors, in contrast to the existing works which can handle either one of them or assume convex problems.  We show that RCPP enjoys sublinear convergence rate for smooth and possibly nonconvex objective functions over general directed networks. Moreover, under the additional Polyak-Łojasiewicz condition, linear convergence rate can be achieved for RCPP. Numerical examples verify the theoretical findings and demonstrate the efficiency, flexibility, and robustness of the proposed algorithm.
	\end{abstract}
	
	
	%
	\IEEEpeerreviewmaketitle
	\section{Introduction}
	%
	%
	%
	%
	\IEEEPARstart{I}{n} this paper,  the following decentralized optimization problem is studied: 
	\begin{equation} \label{problem}
		\min _{\vx\in\RR^p}~  f(\vx):=\frac{1}{n}\sum _{i=1}^n f_i(\vx),
	\end{equation}
	where  $n$ is the number of agents, $\vx$ is the global decision variable, and the local objective function $f_i: \RR^p\rightarrow \RR$ is only held by agent $i$.  The goal is to find an optimal and consensual solution through local computation and local sharing of information in a directed communication network.  Problem \eqref{problem} has found extensive applications in signal processing; see the recent surveys \cite{chang2020distributed,liu2024survey} and the references therein.
	
	In recent years, many decentralized algorithms were proposed to solve \eqref{problem}. The seminal work \cite{Nedic2009distributed} proposed the distributed subgradient descent (DGD) method, where each agent updates its local copy by mixing with the received copies from neighbors in the network and moving towards the local gradient descent direction. However, under a constant step-size, DGD only converges to a neighborhood of the optimal solution. To obtain better convergence results, various works with bias-correction techniques were proposed, including EXTRA \cite{Shi2015Extra}, exact diffusion \cite{yuan2019exact1}, and gradient tracking based methods \cite{Xu2015Augmented,Di2016Next,Nedic2017achieving,Qu2018Harnessing}. These methods achieve linear convergence for minimizing strongly convex and smooth objective functions. While most of the existing works focus on the undirected graph, directed communication topology also draws significant attentions since it is involved in many scenarios. For example, in wireless networks where agents adopt different broadcasting power levels, the communication capability in one direction may not imply the other \cite{Nedic2018Network}. Several distributed methods have been contemplated under the more general directed network topology; see \cite{Nedic2017achieving,Tsianos2012PushSum,Nedic2015Distributed,zeng2017extrapush,xi2017dextra,xi2017add,Xin2018Linear,Xin2020General,Pu2020Robust,Pu2021Push,sun2022distributed} and the references therein. In particular, the push-pull method \cite{Xin2018Linear,Pu2021Push} has been shown to be effective for minimizing both convex and nonconvex objective functions.
	Interested readers may refer to  \cite{Nedic2018Network,yang2019survey} for a comprehensive survey on decentralized algorithms.
	
	In decentralized computation, exchanging complete information between neighboring agents may suffer from the communication bottleneck due to the limited energy and/or bandwidth. Using compression operators is one of promising ways to  reduce the communication expenses and has been widely studied in the literature  \cite{tang2018communication, Koloskova2019ChocoSGD, koloskova*2020decentralized,liu2020linear,Kajiyama2020Linear,kovalev2021linearly,Song2022CPP,Liao2022CGT,Magnusson2020Maintaining,xiong2022quantized}. Most of the works have considered the \emph{relative compression error} assumption, including unbiased compressors \cite{liu2020linear,kovalev2021linearly,Song2022CPP} and contractive biased compressors \cite{Koloskova2019ChocoSGD, koloskova*2020decentralized}, or the unification of them \cite{Liao2022CGT}. Recently, a few works have also considered quantized compression operators with \emph{absolute compression errors}  \cite{Kajiyama2020Linear,Magnusson2020Maintaining,xiong2022quantized}. 
	To explore a unified framework for (contractive biased) relative and absolute compression errors, the work in \cite{michelusi2022finitebit} studied finite-bit quantization.  However, it requires to assume the absolute compression error diminishes exponentially fast for the desired convergence. In \cite{nassif2023quantization}, the unbiased relative compression error was considered together with the absolute compression error, but the algorithm only converges to a neighborhood of the optimum. 
	
	In this paper, we propose a robust compressed push-pull method (RCPP) for decentralized optimization with communication compression over general directed networks. In particular,  the considered algorithm is robust against compression errors that satisfy a more general assumption, which unifies both (contractive biased and unbiased) relative and absolute compression errors. By comparison, the work in \cite{michelusi2022finitebit} considered only contractive biased relative compression error and exponentially diminishing absolute compression error, and the relative compression error studied in \cite{nassif2023quantization} needs to be unbiased.
	
	By employing the dynamic scaling compression technique to deal with the absolute compression error, and using difference compression for transmitting both the decision variables and the gradient trackers, RCPP provably achieves sublinear convergence for smooth and nonconvex objective functions under the general class of compression operators.  The convergence analysis pertains to constructing a linear system of inequalities involving the optimization error, consensus error, gradient tracking error and compression errors. Additionally, an appropriate Lyapunov function is designed for achieving the final results. Furthermore, if the Polyak-Łojasiewicz inequality (PL condition) is satisfied, linear convergence rate is demonstrated.  Such convergence guarantees are stronger than those given in \cite{michelusi2022finitebit,nassif2023quantization}.
	
	The main contribution of this paper is summarized as follows:
	\begin{itemize}
		\item For decentralized optimization with communication compression, we consider a general class of compression operators, which unifies the commonly used relative and absolute compression error  assumptions. Such a condition is most general in the decentralized optimization literature to the best of our knowledge.
		\item We propose a new method called the robust compressed push-pull algorithm that works over general directed networks. Based on the dynamic scaling compression technique  that deals with the absolute compression error, RCPP provably achieves  $\mathcal{O}(\frac{1}{K})$ sublinear convergence rate for minimizing smooth objective functions under the general unified assumption on the compression operators. To our knowledge, RCPP is the first compressed algorithm that provably achieves  $\mathcal{O}(\frac{1}{K})$ convergence rate for nonconvex problems over general directed networks.  The convergence analysis of RCPP pertains to constructing a linear system of inequalities involving various error terms. An appropriate Lyapunov function is then designed for demonstrating the final convergence results.
		\item When the objective functions further satisfy the Polyak-Łojasiewicz (PL) condition, we show RCPP enjoys linear convergence under the general assumption on the compression operators.
		\item Numerical results demonstrate that RCPP is efficient compared to the state-of-the-art methods and robust under various compressors.
	\end{itemize}
	Compared with the conference version \cite{liao2023linearly} which assumes the PL condition, this paper not only provides complete analysis, but also establishes the convergence guarantee for RCPP under general smooth nonconvex objective functions.
	
	\vspace{-3mm}
	\begin{table}[H]
		\centering
		\caption{\small Comparison of related works on decentralized optimization with communication compression for smooth  nonconvex problems.}	
		\label{table:Categorization-NCVX}
		\vspace{-2mm}
			\begin{threeparttable}
				\begin{tabular}{cccc}
					\hline
					References       &relative    & absolute   & graph\\
					\hline
					\cite{koloskova*2020decentralized,Tang2019DeepSqueeze}   &C   &$\times$    &Und \\
					\cite{Zhao2022BEER,Yau2023DoCoM}   &C  &$\times$    &Und \\
					\cite{Huang2023CEDAS}   &U  &$\times$    &Und \\
					\cite{Yi2022Communicationb,Xie2024Communicationefficient,Xu2023Compressed} &G  & $\times$  &Und \\
					\cite{Yi2022Communicationb,Xu2023Compressed} &$\times$  & $\checkmark$  &Und \\
					\cite{Taheri2020Quantized} &C &$\times$    &Di\\
					\textbf{our paper}      &\textbf{G}          &$\pmb{\checkmark}$           &\textbf{Di}    \\
					\hline
				\end{tabular}
			\end{threeparttable}
	\end{table}

	In Table~\ref{table:Categorization-NCVX} and Table~\ref{table:Categorization},  we compare the results of this paper with related works regarding the assumptions on the compression operators, objective functions, graph topologies and convergence guarantees. 
	\vspace{-3mm}
	\begin{table}[H]
		\centering
		\caption{\small Comparison of related works on decentralized optimization with communication compression.}	
		\label{table:Categorization}
		\vspace{-2mm}
			\begin{threeparttable}
				\begin{tabular}{@{}c@{}c@{~~}c@{}c@{}c@{~}c@{}}
					\hline
					References       &relative    & absolute & convergence  & graph & function\\
					\hline
					\cite{Koloskova2019ChocoSGD,koloskova*2020decentralized}   &C\tnote{1}   &$\times$    &sublinear &Und &SVX\tnote{4}\\
					\cite{liu2020linear,kovalev2021linearly}  &U      &$\times$      &linear &Und &SVX\\
					\cite{Zhang2023Innovation,Yau2023DoCoM}     &C     &$\times$      &linear &Und &SVX\cite{Zhang2023Innovation}, PL\cite{Yau2023DoCoM}\\
					\cite{Liao2022CGT,Yi2022Communicationb} &G  & $\times$ & linear &Und &SVX\cite{Liao2022CGT}, PL\cite{Yi2022Communicationb}\\
					\cite{Kajiyama2020Linear,Magnusson2020Maintaining,Yi2022Communicationb} &$\times$   &Q    &linear &Und &SVX\cite{Kajiyama2020Linear,Magnusson2020Maintaining}, PL\cite{Yi2022Communicationb}\\
					\cite{michelusi2022finitebit}  &C   &$\checkmark$   &linear*\tnote{2} &Und &SVX\\
					\cite{nassif2023quantization}   &U  &$\checkmark$  &neighborhood$^3$  &Und &SVX\\	
					\cite{Song2022CPP}  &U       &$\times$         &linear &Di &SVX\\
					\cite{xiong2022quantized} &$\times$            &Q    &linear &Di &SVX\\
					\textbf{our paper}      &\textbf{G}          &$\pmb{\checkmark}$         &\textbf{linear}  &\textbf{Di}     &\textbf{PL}\\	
					\hline
				\end{tabular}
				\begin{tablenotes}
					\footnotesize
					\item[1] `C', `U', and `G' represent contractive biased, unbiased, and general relative compression assumptions, respectively. `Q' represents quantizer. `Und' and `Di' denote undirected and directed graphs, respectively.
					\item[2] * The result has extra requirement, e.g., exponentially decaying error.
					\item[3] The algorithm converges to the neighborhood of the optimal solution.
					\item[4] `SVX' and `PL' represent strongly convex functions and the PL condition, respectively.
				\end{tablenotes}
			\end{threeparttable}
	\end{table}
	The rest of this paper is organized as follows. 
	In  Section~\ref{Sec: Preliminary}, we state the standing assumptions and discuss the compression techniques. The RCPP method is introduced in Section~\ref{Sec: RCPP}. In Section~\ref{Sec: CA}, we establish the convergence results of RCPP under communication compression. Numerical experiments are provided to verify the theoretical findings in Section~\ref{Sec: Simulation}. Finally, conclusions are given in Section~\ref{Sec: Conclusion}.
	
	\vspace{1em}
	{\bf Notation}: A vector is viewed as a column by default. The bold symbols $\mathbf{0}$ and $\mathbf{1}$ represent the $n$-dimensional column vectors with all entries equal to $0$ and $1$, respectively. Each agent $i$ holds a local copy $\vx_i\in\mathbb{R}^p$ of the decision variable and an auxiliary variable $\vy_i\in\mathbb{R}^p$ to track the average gradient. Vectors $\vx_{i}^{k}$ and $\vy_{i}^{k}$ represent their corresponding values at the $k$-th iteration. For simplicity, denote the aggregated variables as
	$\vX := [\vx_1, \vx_2, \ldots, \vx_n]^{\T}\in\mathbb{R}^{n\times p}$, 
	$ \vY := [\vy_1, \vy_2, \ldots, \vy_n]^{\T}\in\mathbb{R}^{n\times p}$. 
	At step $k$, $\vX^{k}$ and $\vY^{k}$ represent their corresponding values. The other aggregated variables $\vH_{x}$, $\vH_{y}$, $\vH_{R}$, $\vH_{C}$, $\vQ_{x}$, $\vQ_{y}$, $\widehat{\vX}$, $\widehat{\vY}$, $\tX$, and $\tY$ are defined similarly. 
	The aggregated gradients are 
	$
	\nabla \vF(\vX):=\left[\nabla f_1(\vx_1), \nabla f_2(\vx_2), \ldots, \nabla f_n(\vx_n)\right]^{\T}\in\mathbb{R}^{n\times p} 
	$  
	and the average of all the local gradients is
	$
	\nabla \overline \vF(\vX) := \frac{1}{n}\vone^\T \nabla \vF(\vX)=\frac{1}{n}\sum_{i=1}^n (\nabla f_i(\vx_i))^\T.
	$ 
	The notations $\|\cdot\|$ and $\|\cdot\|_F$ define the Euclidean norm of a vector and the Frobenius norm of a matrix, respectively. 
	
	The set of nodes (agents) is denoted by $\mathcal{N}=\{1,2,\ldots,n\}$.  A directed graph (digraph) is a pair $\mathcal{G}=(\mathcal{N},\mathcal{E})$, where the edge set $\mathcal{E}\subseteq \mathcal{N}\times \mathcal{N}$ consists of ordered pairs of nodes. If there exists a directed edge from node $i$ to node $j$ in $\mathcal{G}$, or $(i,j)\in\mathcal{E}$, then $i$ is called the parent node, and $j$ is the child node. The parent node can directly transmit information to the child node, but not the other way around. 
	Let $\mathcal{G}_\mathbf{W}=(\mathcal{N},\mathcal{E}_\mathbf{W})$ denote a digraph induced by a nonnegative square matrix $\mathbf{W}=[w_{ij}]$, where $(i,j)\in\mathcal{E}_\mathbf{W}$  if and only if $w_{ji}>0$. In addition, $\mathcal{R}_\mathbf{W}$ is the set of roots of all the possible spanning trees in $\mathcal{G}_\mathbf{W}$.
	
	\section{Problem Formulation}\label{Sec: Preliminary}
	In this section, we first provide the basic assumptions on the communication graphs and the objective functions. Then, we introduce a general assumption on the compression operators to unify both the relative and absolute compression errors.
	
	\subsection{Communication graphs and objective functions}
	Consider the following conditions on the communication graphs among the agents and the corresponding mixing matrices.
	\begin{assumption}\label{Assumption: network}
		The matrices $\vR$ and $\vC$ are both supported by  a strongly connected graph $\mathcal{G} = (\mathcal{N},\mathcal{E})$, i.e., 
		$\mathcal{E}_{\vR} = \{(j,i)\in\mathcal{N}\times\mathcal{N} \big| \vR_{ij} > 0 \} \subset \mathcal{E}$ 
		and 
		$\mathcal{E}_{\vC} = \{(j,i)\in\mathcal{N}\times\mathcal{N} \big| \vC_{ij} > 0 \} \subset \mathcal{E}$. 
		The matrix $\vR$ is row stochastic, and $\vC$ is column stochastic, i.e., $\vR\vone=\vone$ and $\vone^\T\vC=\vone^\T$. In addition, there exists at least one node that is a root of spanning trees for both $\mathcal{G}_{R}$ and $\mathcal{G}_{C^\T}$, i.e., $\mathcal{R}_{\vR} \cap \mathcal{R}_{\vC^\T} \neq \emptyset$.
	\end{assumption}
	\begin{remark}
		Assumption~\ref{Assumption: network} is weaker than requiring both $\mathcal{G}_{\vR}$ and $\mathcal{G}_{\vC}$ to be strongly connected \cite{Pu2021Push}. It implies that $\vR$ has a unique nonnegative left eigenvector $\vu_{R}$ w.r.t. eigenvalue $1$ with $\vu_{R}^\T\vone=n$, and $\vC$ has a unique nonnegative right eigenvector $\vu_{C}$ w.r.t. eigenvalue $1$ such that $\vu_{C}^\T\vone=n$. The nonzero entries of $\vu_{R}$ and $\vu_{C}$ correspond to the nodes in $\mathcal{R}_{\vR}$ and $\mathcal{R}_{\vC^\T}$, respectively. Since $\mathcal{R}_{\vR} \cap \mathcal{R}_{\vC^\T} \neq \emptyset$, we have $\vu_{R}^\T \vu_{C} > 0$.	For a more detailed explanation, please refer to \cite{Pu2021Push}.
	\end{remark}
	For the nonconvex setting, the objective functions are assumed to satisfy the following smoothness assumption.
	\begin{assumption}\label{Assumption: Lipschitz}
		For each agent $i$,  its gradient is $L_i$-Lipschitz continuous, i.e.,
		\begin{equation}\label{L-smooth}
			\|\nabla f_i(\vx)-\nabla f_i(\vx')\|\le L_i \|\vx-\vx'\|,\; \forall\vx,\vx'\in\mathbb{R}^p.
		\end{equation}
	\end{assumption}
	\begin{assumption}\label{Assumption: PL}
		The objective function $f$ satisfies the Polyak-Łojasiewicz inequality (PL condition), i.e.,
		\begin{equation}\label{PL}
			\norm{\nabla f(\vx)}^2\geq 2 \mu \big(f(\vx)-f(\vx^*)\big),
		\end{equation}
		where   $\vx^*$  is an optimal  solution to problem \eqref{problem}. 
	\end{assumption}
	\subsection{A unified compression assumption}\label{Sec: Compression}
	We now present a general assumption on the compression operators which incorporates both relative and absolute compression errors. 	
	\begin{assumption}\label{Assumption:General}
		The  compression operator $\cC \colon \RR^d \to \RR^d$  satisfies  
		\begin{align}\label{def:Ncompressor4}
			\EE_{\cC} \big[\norm{\cC(\vx) -\vx}^2 \big]&\leq C\norm{\vx}^2+\sigma^2,& ~\forall \vx \in \RR^d, 
		\end{align}
		for some constants $C,\sigma^2\ge 0$, and the $r$-scaling of $\cC$ satisfies 
		\begin{align}\label{def:contract4}
			\EE_{\cC} \big[\norm{\cC(\vx)/r -\vx}^2\big] &\leq (1-\delta)\norm{\vx}^2+\sigma^2_{r},& ~\forall \vx \in \RR^d , 
		\end{align}
		for some constants $r>0$, $\delta\in(0,1]$ and $\sigma^2_r\ge 0$.
	\end{assumption}
	Among the compression conditions considered for decentralized optimization algorithms with convergence guarantees, Assumption \ref{Assumption:General} is the weakest to the best of our knowledge. Specifically, if there is no absolute error, i.e., $\sigma^2=\sigma_{r}^2=0$, then Assumption \ref{Assumption:General} degenerates to the assumption in \cite{Liao2022CGT} that unifies the compression operators with relative errors. If there is no relative error, i.e., $C=0$ and $\delta=1$, then the condition becomes the assumption on the quantizers in \cite{Kajiyama2020Linear,xiong2022quantized}. If $C<1$, Assumption \ref{Assumption:General} reduces to the condition in \cite{michelusi2022finitebit}. Therefore, Assumption \ref{Assumption:General} provides a unified treatment for both relative and absolute compression errors.  Moreover, Assumption \ref{Assumption:General} does not require any condition on the high-precision representation of specific compression variables, and the machine precision can be deemed as the absolute compression error.  In the numerical examples, we will introduce the modified composition of quantization and Top-k compression  \cite{liao2023linearly}, which adheres to Assumption \ref{Assumption:General} and allows for fewer communication bits.
	\section{Proposed Robust Compressed Push-Pull Method}\label{Sec: RCPP}
	While Assumption~\ref{Assumption:General} provides a unified condition on the compression operators, new challenges are brought to the algorithm design and analysis. Without proper treatment for the compression errors, the algorithmic performance could deteriorate, particularly due to the absolute error that may lead to compression error accumulation. In this section, we first introduce the dynamic scaling compression technique that copes with the absolute compression error. Then, we propose the RCPP algorithm and discuss its connections with the existing methods.
	\subsection{Dynamic scaling compression technique}
	To tackle the challenge brought by the absolute compression error, we consider the dynamic scaling compression technique \cite{Kajiyama2020Linear}. Using a dynamic parameter $s_k$ associated with the $k$-th iteration, we provide the operator $\cQ(\vx)=s_k\cC(\vx/s_k)$. Then from Assumption~\ref{Assumption:General}, we know 
	\begin{align}
		\nonumber \EE_{\cQ} \big[\norm{\cQ(\vx) -\vx}^2\big]
		=& \EE_{\cC} \big[\norm{s_k\cC(\vx/s_k) -\vx}^2\big]\\
		\nonumber=& s_k^{2}\EE_{\cC} \big[\norm{\cC(\vx/s_{k}) -\vx/s_{k}}^2\big]\\
		\nonumber\leq& s_k^2(C\norm{\vx/s_k}^2+\sigma^2)\\
		=& C\norm{\vx}^2+s_k^2\sigma^2.
	\end{align}
	Similarly, we have
	\begin{align}
		\EE_{\cQ} \big[\norm{\cQ(\vx)/r -\vx}^2  \big] \leq (1-\delta)\norm{\vx}^2+s_k^2\sigma^2_{r}.
	\end{align}
	It should be noted that only $\cC(\vx/s_k)$ needs to be transmitted during the communication process, and the signal recovery is accomplished by calculating $\cQ(\vx)=s_k\cC(\vx/s_k)$ on the receiver's side. By using the dynamic scaling compression technique, the absolute errors can be controlled by decaying the parameter $s_k$.
	\subsection{Proposed robust compressed push-pull method}
	We describe the proposed RCPP method in Algorithm \ref{Alg:RCPPe}.  For simplicity, the matrix form is provided. First, we input the initial decision variables $\vX^{0}$ and initialize the gradient trackers $\vY^{0}=\nabla\vF(\vX^{0})$ along with the auxiliary compression variables $\vH_{x}^0=\vzero$, $\vH_{y}^0=\vzero$, $\vH_{R}^0=\vzero$, and $\vH_{C}^0=\vzero$. Then, we set the total number of iterations $K$, the relative compression error control parameters $\alpha_x$ and $\alpha_y$, the absolute compression error control parameters $\{s_k\}_{k\ge 0}$, and the global consensus parameters $\gamma_x$ and $\gamma_y$. Additionally, we decide $\Lambda=\text{diag}([\lambda_1,\lambda_2,\ldots,\lambda_n])$, where $\lambda_i$ represents the step-size of agent $i$.
	\vspace{-3mm}
	\begin{figure}[htbp]
		\centering
		\begin{minipage}{.99\linewidth}
			\begin{algorithm}[H]
				\caption{A Robust Compressed Push-Pull Method}
				\label{Alg:RCPPe}
				\textbf{Input:} initial values $\vX^{0}$, $\vY^{0}=\nabla\vF(\vX^{0})$, $\vH_{x}^0=\vzero$, $\vH_{y}^0=\vzero$, $\vH_{R}^0=\vzero$, $\vH_{C}^0=\vzero$, number of iterations $K$,  parameters $\alpha_x,\alpha_y$, $\{s_k\}_{k\ge 0}$, $\gamma_x,\gamma_y$, step-sizes $\Lambda=\text{diag}([\lambda_1,\lambda_2,\ldots,\lambda_n])$
				\begin{algorithmic}[1]
					\For{$k=0,1,2,\dots, K-1$}
					\State $\widetilde{\vX}^{k}=\vX^{k}-\Lambda \vY^{k}$ 
					\State $\vC^{k}_{x}=\cC((\widetilde{\vX}^{k}-\vH^{k}_{x})/s_k)$ 
					\State $\widehat{\vX}^{k}=\vH^{k}_{x}+\vQ^{k}_{x}$ $^1$ \footnotetext{$^1$ $\vQ^{k}_{x}$ is the result of  dynamic scaling compression with  $\vQ^{k}_{x}=\cQ(\widetilde{\vX}^{k}-\vH^{k}_{x})=s_k\cC((\widetilde{\vX}^{k}-\vH^{k}_{x})/s_k)=s_k\vC^{k}_{x}$. The operation for $\vQ^{k}_{y}$ is the same.}
					\State $\widehat{\vX}_{R}^{k}=\vH^{k}_{R}+\vR\vQ^{k}_{x}$ \Comment{Communication}
					\State $\vH^{k+1}_{x}=(1-\alpha_x)\vH^{k}_{x}+\alpha_x\widehat{\vX}^{k}$
					\State $\vH_{R}^{k+1}=(1-\alpha_x)\vH^{k}_{R}+\alpha_x\widehat{\vX}_{R}^{k}$
					\State $\vX^{k+1}=\widetilde{\vX}^{k}-\gamma_{x}(\widehat{\vX}^{k}-\widehat{\vX}_{R}^{k})$   
					\State $\widetilde{\vY}^{k}= \vY^{k}+\nabla\vF(\vX^{k+1})-\nabla\vF(\vX^{k})$ 
					\State $\vC^{k}_{y}=\cC((\widetilde{\vY}^{k}-\vH^{k}_{y})/s_k)$ 
					\State $\widehat{\vY}^{k}=\vH^{k}_{y}+\vQ^{k}_{y}$
					\State $\widehat{\vY}_{C}^{k}=\vH^{k}_{C}+\vC\vQ^{k}_{y}$ \Comment{Communication}
					\State $\vH^{k+1}_{y}=(1-\alpha_y)\vH^{k}_{y}+\alpha_y\widehat{\vY}^{k}$
					\State $\vH_{C}^{k+1}=(1-\alpha_y)\vH^{k}_{C}+\alpha_y\widehat{\vY}_{C}^{k}$
					\State $\vY^{k+1}=\widetilde{\vY}^{k}-\gamma_{y}(\widehat{\vY}^{k}-\widehat{\vY}_{C}^{k})$ 
					\EndFor
				\end{algorithmic}
				\noindent\textbf{Output:} $\vX^{K},\vY^{K}$
			\end{algorithm}
		\end{minipage}
	\end{figure}
	
	Lines 2 and  9 represent the local updates for the  decision variables and the gradient trackers, respectively. In Lines 3 and 10, the dynamic scaling compression technique is respectively applied to execute the difference compression between the local updates and the auxiliary variables. The difference compression reduces the relative compression errors \cite{Koloskova2019ChocoSGD,Liao2022CGT}, while the dynamic scaling compression controls the absolute compression errors.   Note that difference compression in Line 10 is also applied for transmitting the gradient trackers, which helps better handle the general relative compression errors, whereas \cite{Song2022CPP} only considered direct compression.
	
	The operator $\cQ$ is a dynamic scaling compressor defined by $\cQ(\vx)=s_k\cC(\vx/s_k)$. 
	The compressed vector  $\cC((\widetilde{\vx}_{i}^{k}-\vh^{k}_{i,x})/s_k)$  is transmitted to the neighbors of agent $i$ and recovered by computing $s_k \cC((\widetilde{\vx}_{i}^{k}-\vh^{k}_{i,x})/s_k)$ after communication, where $\widetilde{\vx}_{i}^{k}$ and $\vh^{k}_{i,x}$ denote agent $i$'s local update and auxiliary variable, respectively. It is worth noting that if the dynamic scaling compression technique is not used, then the absolute compression error would accumulate and significantly impact the algorithm's convergence.  Using an appropriately designed decaying $\{s_k\}$ ensures that the absolute compression error is summable, thereby making the algorithm robust to such an error. This is also the primary distinction compared to the algorithm design of CPP \cite{Song2022CPP}.
	
	In Lines 4 and 11, the decision variables and the gradient trackers are locally recovered, respectively. Lines 5 and 12 represent the communication steps, where each agent mixes the received compressed vectors multiplied by $s_k$.  The variables $\widehat{\vX}_{R}^{k}$ and $\widehat{\vY}_{C}^{k}$ are introduced to store the  aggregated information received from the communication updates. By introducing such  auxiliary variables, there is no need to store all the neighbors' reference points \cite{Koloskova2019ChocoSGD,liu2020linear}. Lines 6-7 and 13-14 update the auxiliary variables, where parameters $\alpha_x,\alpha_y$ control the relative compression errors; see e.g., \cite{Liao2022CGT} for reference. The consensus updates are performed in Lines 8 and 15, where $\gamma_{x},\gamma_{y}$ are the global consensus parameters to guarantee the algorithmic convergence. It is worth noting that RCPP employs the adapt-then-combine (ATC) strategy compared with CPP \cite{Song2022CPP}, which accelerates the convergence \cite{Nedic2017achieving}.
	
	To understand the connection between RCPP and the Push-Pull/AB algorithm \cite{Xin2018Linear, Pu2021Push}, note that we have $\vH_{R}^0=\vR \vH_{x}^0$ and $\vH_{C}^0=\vC \vH_{y}^0$ from the initialization. By induction, it follows  that $\widehat{\vX}_{R}^{k}=\vR \widehat{\vX}^{k}$, $\widehat{\vY}_{C}^{k}=\vC \widehat{\vY}^{k}$ and  $\vH_{R}^k=\vR \vH_{x}^k$, $\vH_{C}^k=\vC \vH_{y}^k$. Recalling Lines 8 and 15 in Algorithm~\ref{Alg:RCPPe}, we have
	\begin{align}
		\nonumber\vX^{k+1}=&\widetilde{\vX}^{k}-\gamma_{x}(\widehat{\vX}^{k}-\widehat{\vX}_{R}^{k})
		=\widetilde{\vX}^{k}-\gamma_{x}(\widehat{\vX}^{k}-\vR \widehat{\vX}^{k})\\
		\label{Xkplus}=&\widetilde{\vX}^{k}-\gamma_{x}(\vI-\vR)\widehat{\vX}^{k}
	\end{align}
	and
	\begin{align}
		\nonumber \vY^{k+1}=&\widetilde{\vY}^{k}-\gamma_{y}(\widehat{\vY}^{k}-\widehat{\vY}_{R}^{k})
		=\widetilde{\vY}^{k}-\gamma_{y}(\widehat{\vY}^{k}-\vC \widehat{\vY}^{k})\\
		\label{Ykplus}=&\widetilde{\vY}^{k}-\gamma_{y}(\vI-\vC)\widehat{\vY}^{k}.
	\end{align}
	If $\tX^{k}$ and $\tY^{k}$ are not compressed, i.e., $\widehat{\vX}^{k}=\tX^{k}$ and $\widehat{\vY}^{k}=\tY^{k}$, then
	\begin{align}
		\nonumber \vX^{k+1}=&\tX^{k}-\gamma_{x} (\vI-\vR)\tX^{k}\\
		=&[(1-\gamma_{x})\vI+\gamma_x \vR](\vX^{k}-\Lambda \vY^{k}),
	\end{align}
	and 
	\begin{align}
		\nonumber &\vY^{k+1}=\tY^{k}-\gamma_{y} (\vI-\vC)\tY^{k}\\
		&~=[(1-\gamma_{y})\vI+\gamma_y\vC](\vY^{k}+\nabla\vF(\vX^{k+1})-\nabla\vF(\vX^{k})).
	\end{align}
	Letting the consensus step-sizes be $\gamma_{x}=1$ and $\gamma_{y}=1$, the above updates recover those in the Push-Pull/AB algorithm \cite{Xin2018Linear, Pu2021Push}. 
	
	In addition, RCPP retains the property of gradient tracking based methods. From Line 15 in Algorithm 1,
	\begin{align}
		\nonumber\vone^\T \vY^{k+1}=&\vone^\T  (\widetilde{\vY}^{k}-\gamma_{y}(\vI-\vC)\widehat{\vY}^{k})\\
		\nonumber =&\vone^\T ( \vY^{k}+\nabla\vF(\vX^{k+1})-\nabla\vF(\vX^{k}))\\
		=&\vone^\T \nabla\vF(\vX^{k+1}),
	\end{align}
	where the second equality is from $\vone^\T (\vI-\vC)=\vzero$, and the last equality is deduced by induction given that $\vY^{0}=\nabla\vF(\vX^{0})$. Define $\oX^{k}=\frac{1}{n}\vu_{R}^{\T}\vX^{k}$	and $\oY^{k}=\frac{1}{n}\vone^{\T}\vY^{k}$.  Once $(\vx_{i}^{k})^\T\rightarrow\oX^{k}$ and $(\vy_{i}^{k})^\T\rightarrow\oY^{k}$, then each agent can track the average gradient, i.e.,  
	\begin{align}
		(\vy_{i}^{k})^\T\rightarrow\oY^{k}=\frac{1}{n}\vone^\T\nabla\vF(\vX^{k})\rightarrow \frac{1}{n}\vone^\T\nabla\vF(\vone\oX^{k}).
	\end{align}
	
	\section{Convergence analysis}\label{Sec: CA}
	In this section, we study the convergence property of RCPP under smooth objective functions  (with or without the PL condition).
	
	For simplicity of notation, denote $\Pi_{R}=\vI-\frac{\vone \vu_{R}^{\T}}{n}$, $\Pi_{C}=\vI-\frac{\vu_{C}\vone^{\T}}{n}$ and $\vX^*=(\vx^*)^\T\in\mathbb{R}^{1\times p}$. Denote $\overline{\lambda}=\frac{1}{n}\vu_{R}^{\T}\Lambda\vu_{C}$, $\hlambda=\max\limits_{i}\{\lambda_i\}$. From Assumption \ref{Assumption: Lipschitz}, the gradient of $f$ is $L$-Lipschitz continuous, where $L=\max{\{L_i\}}$.
	
	The main ideas of the convergence analysis are as follows. First, we bound the optimization error $\Omega_{o}^{k}:=\EE\big[f(\oX^{k})-f(\vX^*)\big]$, consensus error $\Omega_{c}^{k}:=\EE\big[\|\Pi_{R}\vX^{k}\|_R^2\big]$, gradient tracking error $\Omega_{g}^{k}:=\EE\big[\|\Pi_{C}\vY^{k}\|_C^2\big]$, and compression errors $\Omega_{cx}^{k}:=\EE\big[\| \tX^{k}- \vH^{k}_x \|_F^2\big]$ and  $\Omega_{cy}^{k}:=\EE\big[\|\vY^{k}-\vH^{k}_y\|_F^2\big]$ through a linear system of inequalities, where $\|\Pi_{C}\vY^{k}\|_R$ and $\|\Pi_{C}\vY^{k}\|_C$ are specific norms introduced in Lemma \ref{Norm_R_C}. Then, we derive the descent property of the objective function by leveraging the smoothness of the gradient. Finally, the algorithmic convergence is demonstrated using an appropriately designed Lyapunov function.
	\subsection{Some Key Auxiliary results}
	We first introduce two supporting lemmas. 
	\begin{lemma}\label{Norm_R_C}
		There exist invertible matrices  $\widetilde{\vR},\widetilde{\vC}$ such that the induced vector norms $\norm{\cdot}_{R}$ and $\norm{\cdot}_{C}$ satisfy $\norm{\vv}_{R}=\Vert\widetilde{\vR}\vv\Vert$ and $\norm{\vv}_{C}=\Vert\widetilde{\vC}\vv\Vert$, respectively.  In addition, the following properties hold:
		\begin{itemize}
			\item[(1)]
			For any vector $\vv$, we have $\norm{\vv}\leq \norm{\vv}_{R}, \norm{\vv}\leq \norm{\vv}_{C}$, and there exist constants $\delta_{R,2}, \delta_{C,2}$ such that $\norm{\vv}_{R}\leq \delta_{R,2} \norm{\vv},  \norm{\vv}_{C}\leq \delta_{C,2}\norm{\vv}$.
			\item[(2)]
			For any $\gamma_{x},\gamma_{y} \in (0,1]$, we have
			$\norm{\Pi_{R}\vR_{\gamma}}_{R}\leq 1-\theta_{R} \gamma_{x}$ 
			and 
			$\norm{\Pi_{C}\vC_{\gamma}}_{C}\leq 1-\theta_{C} \gamma_{y},$ 
			where $\vR_{\gamma}=\vI-\gamma_x(\vI-\vR)$, $\vC_{\gamma}=\vI-\gamma_y(\vI-\vC)$, and $\theta_{R}$ and $\theta_{C}$ are constants in $(0,1]$. 
			\item[(3)]
			There hold $\norm{\Pi_{R}}_{R}=\norm{\Pi_{C}}_{C}=1$, $\norm{\vR}_{R}\leq 2, \norm{\vC}_{C}\leq 2$, $\norm{\vR_{\gamma}}_{R}\leq 2,\norm{\vC_{\gamma}}_{C}\leq 2$ and $\norm{\vI-\vR}_{R}\leq 3, \norm{\vI-\vC}_{C}\leq 3$. 
		\end{itemize}
	\end{lemma}
	\begin{proof}
		See the supplementary material of \cite{Song2022CPP}.
	\end{proof}
	\begin{lemma}\label{Lem:Yk}
		For $\norm{\vY^{k}}_{F}^{2}$, we have
		\begin{equation}\label{YkNormLem}
			\begin{aligned}
				\norm{\vY^{k}}_{F}^{2}
				\leq& 3\norm{ \Pi_{C}\vY^{k} }_{C}^{2}
				+\frac{3\norm{ \vu_{C} }^{2}}{n} L^2 \norm{ \Pi_{R}\vX^{k} }_{R}^{2} \\
				&+3 \norm{\vu_{C}}^{2}\norm{\nabla f(\oX^{k})}^{2}.
			\end{aligned}
		\end{equation}
	\end{lemma}
	\begin{proof}
		See Appendix \ref{Pf:Yk}.
	\end{proof}
	\begin{remark}
		The results in Lemma \ref{Norm_R_C} are mainly used for handling the mixing matrices. Lemma \ref{Lem:Yk} bounds the gradient trackers to simplify the derivation.
	\end{remark}
	The following lemma introduces the key linear system of inequalities based on the previously established supporting lemmas. 
	\begin{lemma}\label{Lem:RCPP}
		Suppose Assumptions \ref{Assumption: network}, \ref{Assumption: Lipschitz} and \ref{Assumption:General} hold 
		and   $\hlambda\leq \min\big\{\frac{1}{6L},\frac{1}{6\sqrt{C}L}\big\}$. 
		Then we have
		\begin{align}\label{ineq:LMI}
			\vw^{k+1}\leq \vA \vw^{k} + \vb\norm{\vY^{k}}_{F}^{2}+\vzeta^{k},
		\end{align}
		where 
		$
		\vw^k =	\big[ 
		\Omega_{c}^{k}, 
		\Omega_{g}^{k}, 
		\Omega_{cx}^{k}, 
		\Omega_{cy}^{k}\big]^{\T},
		$ 
		$\vzeta^{k}=s_k^{2}\cdot\big[\zeta_{c},\zeta_{g},\zeta_{cx},\zeta_{cy}\big]^\T$. The scalars $\zeta_{c},\zeta_{g},\zeta_{cx},\zeta_{cy}$, which relate to the absolute compression errors, along with  the elements of the matrix $\vA$ and the vector $\vb$, are the corresponding coefficients  of the inequalities in Appendix \ref{Pf:LMI}.
	\end{lemma}
	\begin{proof}
		See Appendix \ref{Pf:LMI}.
	\end{proof}
	\begin{remark}
		The absolute compression errors are included in $\vzeta^{k}$ and can be controlled by $s_k^{2}$. The linear system of inequalities   is the key for the convergence analysis which establishes the relationship among the consensus error, gradient tracking error and compression errors.
	\end{remark}
	The following descent lemma comes from the smoothness of the gradients and will be used for proving the main result together with Lemma \ref{Lem:RCPP}.
	\begin{lemma}\label{Lem:descent}
		Suppose Assumption \ref{Assumption: Lipschitz} ,  $\overline{\lambda}\leq\frac{1}{L}$ and $\overline{\lambda}\geq M \hlambda$ for some $M>0$ hold, we have
		\begin{equation}\label{descent}
			\begin{aligned}
				&f(\oX^{k+1})\leq f(\oX^{k})-\frac{M\hlambda}{2}   \norm{\nabla f(\oX^{k})}^2 \\
				&\qquad+E_1M\hlambda  L^2\norm{\Pi_{R}\vX^{k}}_{R}^{2}
				+E_2M\hlambda\norm{\Pi_{C}\vY^{k}}_{C}^{2}, 
			\end{aligned}
		\end{equation}
		where  $E_1=\frac{\norm{\vu_{R}} \norm{\vu_{C}}}{n^2M}$ and $E_2=\frac{\norm{\vu_{R}}^{2}}{n^2 M^2}$. 
	\end{lemma}
	\begin{proof}
		See Appendix \ref{Pf:descent}.
	\end{proof}
	Based on the preceding two lemmas, we provide below a key lemma for proving the main convergence theorems.
	\begin{lemma}\label{Lem:Key}
		Suppose Assumptions \ref{Assumption: network}, \ref{Assumption: Lipschitz} and \ref{Assumption:General} hold, the scaling parameters $\alpha_x, \alpha_y \in (0, \frac{1}{r}]$, $\olambda\geq M \hlambda$ for some $M>0$, 
		and the consensus step-sizes $\gamma_x,\gamma_y$ and the maximum step-size $\hlambda$ satisfy
		\begin{align*}
			&\hlambda \leq \Big\{\frac{1}{6},\frac{1}{6\sqrt{C}},\frac{1}{M},
			\frac{\theta_{C}}{\sqrt{54e_1 } },
			\frac{e_3\gamma_{x}}{\sqrt{96E}\norm{\vu_{C}}},
			\frac{e_4\theta_{C}\gamma_{y}}{\sqrt{48e_1}}
			\Big\}\frac{1}{L},\\
			&\gamma_{x} \leq \Big\{1, 
			\frac{e_5\alpha_x r \delta}{\sqrt{C}},
			\frac{\sqrt{M\norm{\vu_{C}}}}{\sqrt{108\norm{\vu_{R}}e_1 }}\theta_{C}\gamma_y
			\Big\},\\
			&\gamma_{y} \leq \Big\{1, 
			\frac{\theta_{C}e_2}{432 e_1},
			\frac{\theta_{C}(\alpha_y r \delta)^2}{432 e_2 C  },
			\frac{\alpha_y r \delta }{\sqrt{1728 C }}
			\Big\}.
		\end{align*}
		Then we have
		\begin{align}\label{ineq:Key}
			\nonumber	&\Omega_{o}^{k+1}+\frac{E M\hlambda}{\beta} V^{k+1}
			\leq \Omega_{o}^{k}-\frac{M\hlambda}{4}   \norm{\nabla f(\oX^{k})}^2\\
			\nonumber	&\quad+\frac{E M\hlambda}{\beta}\Big[
			\Big( 1-\frac{3\theta_{R}\gamma_x}{32}\Big) L^2 \Omega_{c}^{k}
			+\Big( 1-\frac{\theta_{C}\gamma_y }{8} \Big) A\Omega_{g}^{k}\\
			&+\Big(1-\frac{\alpha_x r \delta}{4}\Big)B\Omega_{cx}^{k}
			+\Big(1-\frac{\alpha_y r \delta}{16}
			\Big)D\Omega_{cy}^{k}
			\Big]
			+s_k^{2}\tilde{\zeta_0},
		\end{align}
		where we denote 
			$$\beta=\frac{\theta_{R}\gamma_{x}}{8},V^{k}=L^2\Omega_{c}^{k}+A\Omega_{g}^{k}+B\Omega_{cx}^{k}+D\Omega_{cy}^{k},$$ 
					$$\tilde{\zeta_0}=\zeta_0\frac{E M\hlambda}{\beta}, \zeta_0=L^2\zeta_{c}+A\zeta_{g}+B\zeta_{cx}+D\zeta_{cy}$$
					for simplicity, and the other parameters are given by
					$$A=\frac{\theta_{C}\gamma_{y}\theta_{R}}{108 e_1\gamma_{x}},   
					B=\frac{L^2 \alpha_x r \delta \theta_{R}}{1296  \gamma_x},$$ 
					$$D=\frac{\theta_{C}\gamma_{y} \alpha_y r \delta \theta_{R}}{108 e_2 \gamma_x}\leq\frac{  \alpha_y r \delta \theta_{R}}{108 e_2  \gamma_x},    
					E=\frac{\norm{\vu_{R}} \norm{\vu_{C}}}{n^2M},$$  
					$$e_1=2\delta_{C,2}^2(1+18C),   
					e_2=108C+112,$$  
					$$e_3=\min\Big\{
					\frac{1}{2\delta_{R,2}}\frac{\theta_{R}}{1-\theta_{R}\gamma_{x}},1
					\Big\},$$  
					$$e_4=\min\Big\{
					\frac{\theta_{R}}{36(1-\theta_{R}\gamma_{x})},
					\frac{1}{ 3\sqrt{ 3}}
					\Big\},$$ 
					$$e_5=\min\Big\{\frac{\theta_{R} }{432\sqrt{2}\delta_{R,2}},
					\frac{1}{72}
					\Big\}.$$ 
	\end{lemma}
	\begin{proof}
		See Appendix \ref{Pf:Key}.
	\end{proof}
	\subsection{General nonconvex case: sublinear convergence}
	Define a Lyapunov function $\Omega^{k}=\Omega_{o}^{k}+\frac{E M\hlambda}{\beta} V^{k}$. Based on the previous lemmas, we first demonstrate the following  $\mathcal{O}\left(\frac{1}{K}\right)$  convergence rate for RCPP under smooth nonconvex objective functions.  
	\begin{theorem}\label{Thm:RCPP-NCVX}
		Suppose Assumptions \ref{Assumption: network}, \ref{Assumption: Lipschitz} and \ref{Assumption:General} hold, the scaling parameters $\alpha_x, \alpha_y \in (0, \frac{1}{r}]$, $\olambda\geq M \hlambda$ for some $M>0$, $s_k^2=a_0a^k$,  $a\in(0,1),a_0>0$, 	and the consensus step-sizes $\gamma_x,\gamma_y$ and the maximum step-size $\hlambda$ satisfy the same condition as in Lemma \ref{Lem:Key}. 
		Then, it holds that
		\begin{align}\label{sublinear}
			\frac{1}{K}\sum_{k=0}^{K-1}\mathbb{E}\Big[\norm{\nabla f(\oX^{k})}^2\Big]\leq \frac{4\Omega^{0}}{M \hlambda K}+\frac{4E \tilde{c}\zeta_0 }{\beta K},
		\end{align}
		where $\tilde{c}=\frac{a_0}{1-a}$.
	\end{theorem}
	\begin{proof}
		See Appendix \ref{Pf:RCPP-NCVX}.
	\end{proof}
	\begin{remark}
		From Theorem \ref{Thm:RCPP-NCVX}, the convergence rate of RCPP is affected by the compression errors. 
		Specifically, the first term on the right hand side on \eqref{sublinear} is only related to the relative compression error. From the upper bounds of $\hlambda$ and $\gamma_{y}$, we have $\gamma_{y}\sim\mathcal{O}\left(\frac{\delta^2}{C^2}\right)$ and $\hlambda\sim\mathcal{O}\left(\frac{\gamma_{y}}{\sqrt{C}}\right)=\mathcal{O}\left(\frac{\delta^2}{C^{2.5}}\right)$, and thus the first term has a dependency of $\mathcal{O}\left(\frac{C^{2.5}}{\delta^2}\right)$ regarding the relative compression error.  For the second  term of \eqref{sublinear}, we have
		\begin{equation*}
			\begin{aligned}
				\frac{4E \tilde{c}\zeta_0 }{\beta K}
				=&\frac{4E}{K} \frac{a_0}{1-a} \frac{8}{\theta_{R}\gamma_{x}}  (L^2\zeta_{c}+A\zeta_{g}+B\zeta_{cx}+D\zeta_{cy}) \\
				=&\frac{32E}{K} \frac{a_0}{1-a}
				\Big(\frac{18L^2\delta_{R,2}^2\sigma^2}{\theta_{R}^{2}\gamma_{x}}\\
				&+\frac{\gamma_{y}}{6 e_1\gamma_{x}^{2}}\big(6\delta_{C,2}^2(1+18C)L^2+\delta_{C,2}^2\big)\sigma^2\\
				&+\frac{L^2  }{648  \gamma_x^{2}}\big((\alpha_x r)^2 \delta \sigma^2_{r}+9\sigma^2(1/(4L^2)+ 18 )\big)\\
				&+\frac{\theta_{C}\gamma_{y} }{108 e_2 \gamma_x^{2}}\big(2(\alpha_y r)^{2}\delta \sigma^2_{r}+3\sigma^2(18+9e_2L^2)
				\big)
				\Big).\\
			\end{aligned}
		\end{equation*}
	Thus, by designing an appropriate dynamic scaling parameter $s_k$, e.g., $a_{0}\propto \gamma_{x}^{2}$, the second term will primarily depend on the absolute compression error. Consequently, we can effectively separate the influence of relative and absolute compression errors  and achieve the $\mathcal{O}\left(\frac{1}{K}\right)$ convergence rate.
	\end{remark}
	\subsection{ PL condition: linear convergence}
	In this part, we show the convergence rate of RCPP can be improved under the additional PL condition. Particularly, the following theorem demonstrates the linear convergence rate of RCPP for minimizing smooth objective functions satisfying the PL condition.
	\begin{theorem}\label{Thm:RCPP-PL}
		Suppose Assumptions \ref{Assumption: network}, \ref{Assumption: Lipschitz}, \ref{Assumption: PL} and \ref{Assumption:General} hold, the scaling parameters $\alpha_x, \alpha_y \in (0, \frac{1}{r}]$, $\olambda\geq M \hlambda$ for some $M>0$,  $s_k^2=c_0c^k$, and the consensus step-sizes $\gamma_x,\gamma_y$ and the maximum step-size $\hlambda$ satisfy the same condition as in Lemma \ref{Lem:Key}. 
		Then, the optimization error $\Omega_{o}^{k}$ and the consensus error $\Omega_{c}^{k}$ both converge to $0$ at the linear rate  $\mathcal{O}(c^{k})$, where $c\in(\tilde{\rho},1)$, $\tilde{\rho}=\max\{1-\frac{1}{2}M\hlambda\mu,1-\frac{\theta_{R}\gamma_x}{16},1-\frac{\theta_{C}\gamma_y}{8},1-\frac{\alpha_x r \delta}{4},1-\frac{\alpha_y r \delta}{16}\}$.
	\end{theorem}
	\begin{proof}
		See Appendix \ref{Pf:RCPP-PL}.
	\end{proof}
	\begin{remark}
		The linear convergence rate is related to $\tilde{\rho}$. From the definition of $\tilde{\rho}$, we find that larger step-size $\hlambda$, larger consensus step-sizes $\gamma_{x},\gamma_{y}$, and smaller compression rate (larger $C$ and smaller  $\delta$) lead to faster convergence. 
	\end{remark}
	\begin{remark}
		It is worth nothing that the linear convergence of RCPP does not depend on the decaying assumption of the absolute compression error as in \cite{michelusi2022finitebit}. 
	\end{remark}
	\section{Numerical Examples}\label{Sec: Simulation}
	In this section, we consider the following logistic regression problem with convex and nonconvex  regularization  respectively to confirm the theoretical  findings \cite{Song2022CPP,michelusi2022finitebit,Zhao2022BEER,Yi2022Communicationb}: 
	\begin{align}\label{Logistic Regression} 
		\min_{x\in \mathbb{R}^{p}}f(x)=\frac{1}{n}\sum_{i=1}^n f_i(x)=\frac{1}{n}\sum_{i=1}^n \big(h_i(x)+\frac{\rho}{2}R(x)\big),
	\end{align}
	where $\rho>0$ is a penalty parameter and $R(x)$ is the regularizer. The objective functions $h_i(x)$  are given by 
	\begin{align}
		h_i(x)=\frac{1}{J}\sum_{i=1}^{J}\ln(1+\exp(-v_{i,j}u_{i,j}^{\T}x)),
	\end{align}
	where $\{(u_{i,j},v_{i,j})\}_{j=1}^{J}$ is the local data stored in agent $i$, $u_{i,j}\in\mathbb{R}^p$ represents the features and  $v_{i,j}\in\mathbb{R}$ represents the labels. 
	The regularization term $R(x)$ is given by $R(x)=\norm{x}^2$ for the convex case, and $R(x)=\sum_{t=1}^{p}\frac{x^2[t]}{1+x^2[t]}$ for the nonconvex case, where $x[t]$ is the $t$-th element of $x$.
	
	The logistic regression model parameter is generated by $x\sim\mathcal{N}(\vzero_{p},\vI_{p})$. The features are generated by 
	$u_{i,j}\sim \mathcal{N}(\vzero_{p},\sigma^2\vI_{p})$. The labels $v_{i,j}$ are set to $1$ if $z_{i,j}\leq\frac{1}{1+\exp(-u_{i,j}^{\T}x)}$ and $-1$ otherwise, where $z_{i,j}$ is uniformly and randomly taken from $z_{i,j}\sim\mathcal{U}(0,1)$. In the experiments, we let $p=500$, $n=100$, $J=10$, $\rho=0.01$, $\sigma=1$. The initial decision variables are randomly generated within $[0,1]^p$. The parameters are hand-tuned to achieve the best performance for each algorithm.
	
	For decentralized optimization over directed networks, we compare the performance of RCPP against CPP \cite{Song2022CPP} and QDGT \cite{xiong2022quantized}. The row-stochastic and column-stochastic weights are randomly generated. To save communication, an adaptation from the $b$-bits $\infty$-norm quantization compressor in \cite{liu2020linear} is used, which is given by
	\begin{align}\label{Quant}
		\cQ(\vx)=\frac{\phi(\|\vx\|_{\infty})} {2^{b-1}} \text{sign}(\vx)  \odot \left \lfloor{\frac{2^{b-1}|\vx|}{\|\vx\|_{\infty}}+\vu}\right \rfloor,
	\end{align}
	where $\text{sign}(\vx)$ is the sign function, $\odot$ is the Hadamard product, $|\vx|$ is the element-wise absolute value of $\vx$, and $\vu$ is a random perturbation vector uniformly distributed in $[0, 1]^{p}$. 
	
	\begin{figure}[htbp]
		\centering
		\subfigure[]{ \includegraphics[scale=0.49]{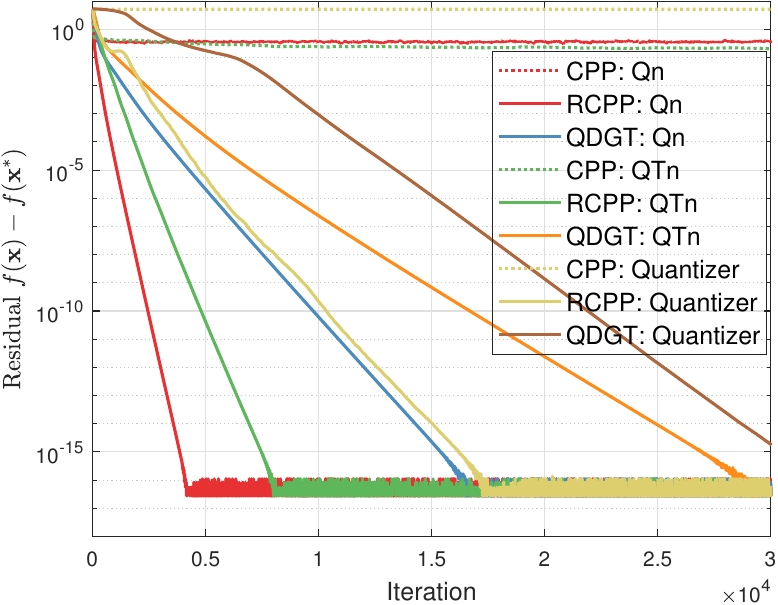} }%
		\hfill
		\subfigure[]{ \includegraphics[scale=0.49]{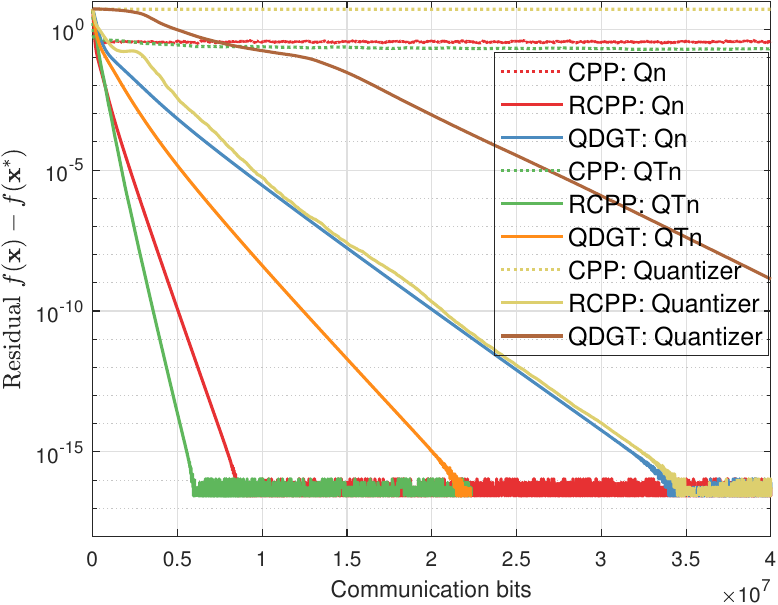} }
		\caption{Residuals $\EE\big[f(\ox^{k})-f(\vx^*)\big]$ against the number of iterations and communication bits respectively for CPP, RCPP and QDGT under different compression methods in the convex case.}
		\label{Fig1}
	\end{figure}
	
	Compared with the original compressor in \cite{liu2020linear}, the modified compression operator utilizes a  mapping $\phi(\|\vx\|)$, where $\phi(\|\vx\|)$ is a random variable defined such that $\phi(\|\vx\|)=\lfloor \|\vx\| \rfloor+1$ with probability $\|\vx\|-\lfloor \|\vx\| \rfloor$ and $\phi(\|\vx\|)=\lfloor \|\vx\| \rfloor$ otherwise. By incorporating $\phi(\|\vx\|)$, only a dynamic number of finite bits, approximately $\log_{2}(\lfloor \|\vx\| \rfloor+1)+1$ bits, are required to be transmitted for nonzero norms. 
	The quantization with the new mapping $\phi(\|\vx\|)$ is denoted as Qn, while QTn represents the composition of quantization and Top-k compression using the same operation. 
	Note that these compression operators yield absolute compression errors and adhere to Assumption \ref{Assumption:General}, whereas QTn does not meet previously established conditions for compression operators. For the experiments conducted in this paper,  the parameters $b$ and $\mathrm{k}$ are set to $b=2$ and $\mathrm{k}=10$ for the quantization and Top-k, respectively. In addition to Qn and QTn, we also investigate the quantizer compression in \cite{Kajiyama2020Linear,xiong2022quantized} which satisfies the absolute compression error assumption. The quantized level is $1$, i.e., the quantized values are from the set $\{-1,0,1\}$.
	
	Next, we evaluate the algorithms' efficiency for both convex and nonconvex cases.
	\subsection{Convex case}
	In Fig. \ref{Fig1}(a), we compare the residuals of CPP, RCPP and QDGT against the number of iterations. It is evident that the performance of CPP deteriorates due to the presence of the absolute compression error. In addition, RCPP consistently outperforms QDGT under different compression methods. 
	
	Then, we further examine the performance of the algorithms in terms of the communication bits. From Fig. \ref{Fig1}(b), we observe that RCPP performs better than the other methods under different compression methods. Additionally, RCPP with QTn achieves the highest communication efficiency. This suggests that, by considering Assumption \ref{Assumption:General} which provides more flexibility in choosing the compression operators, RCPP may achieve superior performance under a specific choice of compressor with fewer communication requirements, while the chosen compressor  may not satisfy the previous assumptions.
	
	\subsection{Nonconvex case}
	\begin{figure}[htbp]
		\centering
		\subfigure[]{ \includegraphics[scale=0.49]{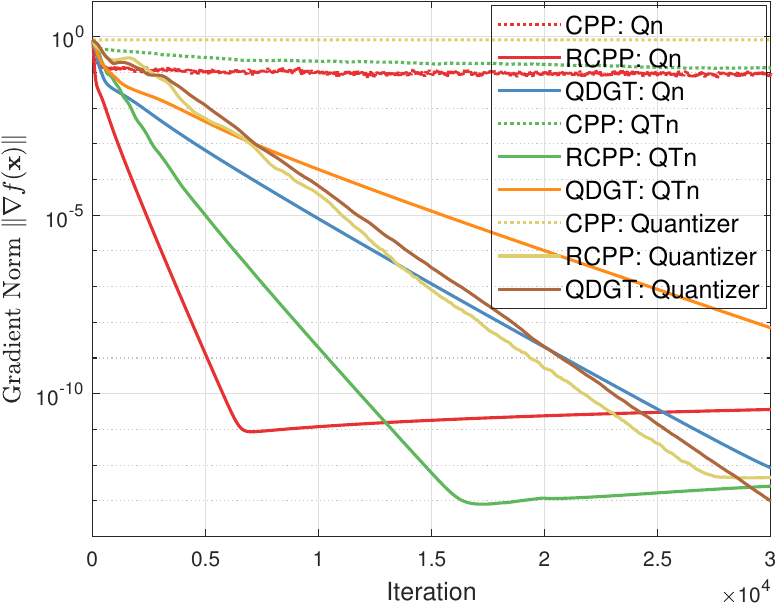} }%
		\hfill
		\subfigure[]{ \includegraphics[scale=0.49]{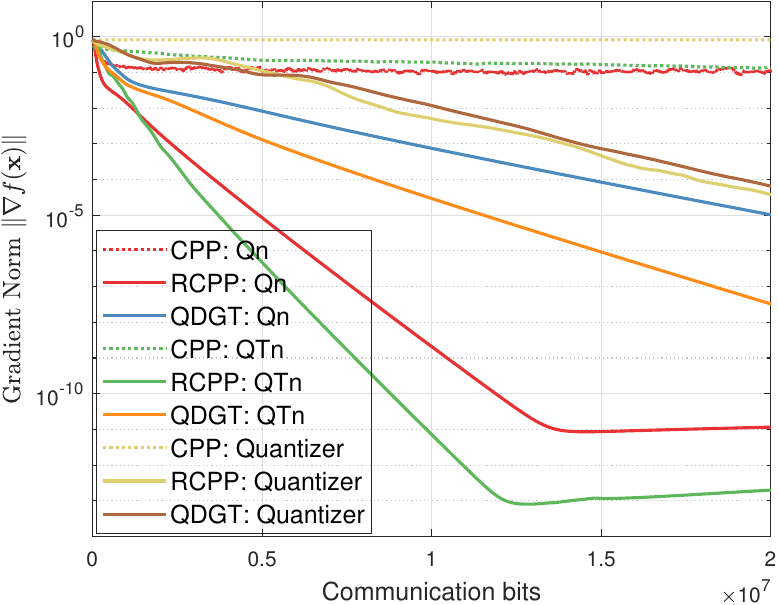} }
		\caption{Gradient norm $\norm{\nabla f(\ox^{k})}$ against the number of iterations and communication bits respectively for CPP, RCPP and QDGT under different compression methods in the nonconvex case.}
		\label{Fig2}
	\end{figure}
	In the nonconvex case, the  gradient norm is used as a metric to evaluate the algorithmic performance. From Fig. \ref{Fig2}(a) and (b), we also find RCPP outperforms CPP and QDGT. In terms of the communication efficiency, RCPP with QTn continues to achieve the best performance. This further demonstrates the advantage of considering the unified assumption, i.e., Assumption \ref{Assumption:General}.

	\section{Conclusions}\label{Sec: Conclusion}
	This paper considers decentralized optimization with communication compression over directed networks. Specifically, we consider a general class of compression operators that allow both relative and  absolute compression errors. For smooth objective functions, we propose a robust compressed push-pull algorithm, termed RCPP. The algorithm is shown to converge sublinearly for general nonconvex objectives and achieve linear convergence rate under an additional PL condition. Numerical results demonstrate that RCPP is efficient and robust to various compressors.
	

\input{bib.tex}
	
	\appendix
	\section{Proofs}\label{Appendix1}
	\subsection{Supplementary Lemmas}
	We here give two lemmas that will be frequently used in the proof of Lemma \ref{Lem:Yk} and for the linear system of inequalities in Lemma \ref{Lem:RCPP}.
	\begin{lemma}\label{lem:UV}
		For  $\vU,\vV\in \RR^{n\times p}$ and any constant $\tau>0$,   we have the following inequality:
		{\small
		\begin{align}
			\|\vU+\vV\|^2\leq (1+\tau)\|\vU\|^2 + (1+\frac{1}{\tau})\|\vV\|^2.
		\end{align}}\normalsize
		In particular, taking $1+\tau=\frac{1}{1-\tau_1}, 0<\tau_1<1$, we have
		{\small
		\begin{align}\label{ineq:UV1}
			\|\vU+\vV\|^2\leq \frac{1}{1-\tau_1} \|\vU\|^2 + \frac{1}{\tau_1}\|\vV\|^2
		\end{align}}\normalsize
		In addition, for any $\vU_1,\vU_2,\vU_3 \in \RR^{n\times p}$, we have 
		{\small
		\begin{align}\label{ineq:UV3}
			\|\vU_1+\vU_2+\vU_3\|^2\leq 3\|\vU_1\|^2 + 3\|\vU_2\|^2+3\|\vU_3\|^2.
		\end{align}}\normalsize
	\end{lemma}
	\begin{lemma}\label{lem:tau}
		For any $\tau\in\RR_{+}$, there holds
		{\small
		\begin{align}
			(1+\frac{\tau}{2})(1-\tau)\leq 1-\frac{\tau}{2}
		\end{align}}\normalsize
	\end{lemma}	
	
	\subsection{Some Algebraic Results}
	Before deriving the linear system of inequalities in Lemma \ref{Lem:RCPP}, we need some preliminary results on $\oX^{k+1}$, $\Pi_{R}\vX^{k+1}$, and $\Pi_{C}\vY^{k+1}$. From some simple algebraic operations, we know $\Pi_{R} \vR=\vR\Pi_{R}=\vR-\frac{\vone \vu_{R}^{\T}}{n}$, $\Pi_{C} \vC=\vC\Pi_{C}=\vC-\frac{\vu_{C}\vone ^{\T}}{n}$, $\Pi_{R}\vX^{k}=\vX^{k}-\vone \oX^{k}$, $\Pi_{C}\vY^{k}=\vY^{k}- \vu_{C}\oY^{k}$, $(\vI-\vR)\Pi_{R}=\Pi_{R}(\vI-\vR)=\vI-\vR$, and $\Pi_{C}(\vI-\vC)=(\vI-\vC)\Pi_{C}=\vI-\vC$.
	With slight notation abuse, the gradient $\nabla f(\oX^{k})$ is regarded as the row vector.
	
	First, from the equivalent formula \eqref{Xkplus}, we know
	{\small
	\begin{align}\label{oXk}	
		\nonumber\oX^{k+1}=&\frac{1}{n}\vu_{R}^{\T}\vX^{k+1}\\
		\nonumber=&\frac{1}{n}\vu_{R}^{\T}\tX^{k}-\gamma_x \frac{1}{n}\vu_{R}^{\T} (\vI-\vR)\widehat{\vX}^{k}\\
		\nonumber	=&\oX^{k}-\frac{1}{n}\vu_{R}^{\T}\Lambda \vY^{k}\\
		\nonumber	=&\oX^{k}-\frac{1}{n}\vu_{R}^{\T}\Lambda (\vY^{k}-\vu_{C}\oY^{k}+\vu_{C}\oY^{k})\\
		\nonumber	=&\oX^{k}-\overline{\lambda} \oY^{k}-\frac{1}{n}\vu_{R}^{\T}\Lambda \Pi_{C}\vY^{k} \\
		\nonumber=&\oX^{k}-\overline{\lambda}\nabla f(\oX^{k})+\overline{\lambda}(\nabla f(\oX^{k})-\oY^{k}) \\
		&-\frac{1}{n}\vu_{R}^{\T}\Lambda \Pi_{C}\vY^{k},
	\end{align}}\normalsize
	where $\oX^{k}=\frac{1}{n}\vu_{R}^{\T}\vX^{k}$, $\overline{\lambda}=\frac{1}{n}\vu_{R}^\T\Lambda \vu_{C}$ and the fact $\vu_{R}^{\T} (\vI-\vR)=\vzero$ is used in the third equality.
	
	Second, for $\Pi_{R}\vX^{k}=\vX^{k}-\vone \oX^{k}$, we have
	{\small
	\begin{align}
		\nonumber	&\Pi_{R}\vX^{k+1}\\
		\nonumber=&\Pi_{R}\tX^{k}-\gamma_x \Pi_{R} (\vI-\vR)\widehat{\vX}^{k}\\
		\nonumber	=&(\vI-\gamma_x(\vI-\vR))\Pi_{R}\tX^{k}+\gamma_x  (\vI-\vR)\Pi_{R}(\tX^{k}-\widehat{\vX}^{k})\\
		\nonumber	=&\Pi_{R}\vR_{\gamma}\Pi_{R}\vX^{k}-\Pi_{R}\vR_{\gamma}\Lambda\vY^{k}\\
		\label{PiRXk}&+\gamma_x  (\vI-\vR)(\tX^{k}-\widehat{\vX}^{k}),
	\end{align}}\normalsize
	where $\vR_{\gamma}=\vI-\gamma_x(\vI-\vR)$ and the relation 
	$\Pi_{R}\vR_{\gamma}\Pi_{R}=\Pi_{R}\vR_{\gamma}=\vR_{\gamma}\Pi_{R}$ is adopted for the third equality.
	
	Finally, for $\Pi_{C}\vY^{k}=\vY^{k}-\vu_{C} \oY^{k}$, we have
	{\small
	\begin{align}
		\nonumber \Pi_{C}\vY^{k+1}=&\Pi_{C}(\tY^{k}-\gamma_{y}(\vI-\vC)\widehat{\vY}^{k})\\
		\nonumber =&(\vI-\gamma_{y}(\vI-\vC))\Pi_{C}\tY^{k}+\gamma_{y}(\vI-\vC)(\tY^{k}-\widehat{\vY}^{k})\\
		\nonumber=&\Pi_{C}\vC_{\gamma}\Pi_{C}\vY^{k}+\Pi_{C}\vC_{\gamma}(\nabla\vF(\vX^{k+1})-\nabla\vF(\vX^{k}))\\
		\label{PiCYk}&+\gamma_{y}(\vI-\vC)(\tY^{k}-\widehat{\vY}^{k}),
	\end{align}}\normalsize
	where $\vC_{\gamma}=\vI-\gamma_y(\vI-\vC)$ and the relation 
	$\Pi_{C}\vC_{\gamma}\Pi_{C}=\Pi_{C}\vC_{\gamma}=\vC_{\gamma}\Pi_{C}$ is adopted for the third equality.
	
	\subsection{Proof of Lemma \ref{Lem:Yk}}\label{Pf:Yk}
	Based on Lemma \ref{lem:UV}, we obtain
	{\small
	\begin{align}\label{YkNorm}
		\nonumber&\norm{\vY^{k}}_{F}^{2}\\
		\nonumber=&\norm{\vY^{k}-\vu_{C}\oY^{k}+\vu_{C}(\oY^{k}-\nabla f(\oX^{k}))+\vu_{C}\nabla f(\oX^{k})}_{F}^{2}\\
		\nonumber	\leq& 3\norm{ \vY^{k}-\vu_{C}\oY^{k} }_{F}^{2}+3\norm{ \vu_{C}(\oY^{k}-\nabla f(\oX^{k}))}_{F}^{2}\\
		\nonumber&+3\norm{\vu_{C}\nabla f(\oX^{k}) }_{F}^{2}\\
		\nonumber\leq& 3\norm{ \Pi_{C}\vY^{k} }_{C}^{2}+\frac{3\norm{ \vu_{C} }^{2}}{n} L^2 \norm{ \Pi_{R}\vX^{k} }_{R}^{2} \\
		&+3 \norm{\vu_{C}}^{2}\norm{\nabla f(\oX^{k})}^{2},
	\end{align}}\normalsize
	where we use the fact that
	{\small
	\begin{align}
		\nonumber	&\norm{\nabla f(\oX^{k})-\oY^{k}}
		\leq \frac{1}{n}\norm{\vone^{\T}\nabla\vF(\vone \oX^{k})-\vone^{\T}\nabla\vF( \vX^{k})}\\
		\leq&\frac{L}{\sqrt{n}}\norm{\vX^{k}-\vone \oX^{k}}_{F}\leq \frac{L}{\sqrt{n}}\norm{\Pi_{R}\vX^{k}}_{R}.
	\end{align}}\normalsize

	\subsection{Proof of Lemma \ref{Lem:RCPP}}\label{Pf:LMI}
	For simplicity, denote $\cF^k$ as the $\sigma$-algebra generated by $\{\vX^0,\vY^0,\vX^1,\vY^1,\cdots,\vX^{k},\vY^{k}\}$, and define $\EE[ \cdot |\cF^k]$ as the conditional expectation with respect to the compression operator given $\cF^k$.  Below, we prove the four inequalities in \eqref{ineq:LMI} respectively. 
	\subsubsection{First inequality}
	From \eqref{PiRXk}, we have
	{\small
	\begin{align}\label{PiRXkNorm0}
		\nonumber	&\norm{\Pi_{R}\vX^{k+1}}_{R}^{2}\\
		\nonumber	=&\norm{\Pi_{R}\vR_{\gamma}\Pi_{R}\vX^{k}-\Pi_{R}\vR_{\gamma}\Lambda\vY^{k}+\gamma_x  (\vI-\vR)(\tX^{k}-\widehat{\vX}^{k})}_{R}^{2}\\
		\nonumber	\leq&(1-\theta_{R}\gamma_x)\norm{\Pi_{R}\vX^{k}}_{R}^{2}
		+\frac{2}{\theta_{R}\gamma_x}\Big( (1-\theta_{R}\gamma_x)^{2}\norm{\Lambda\vY^{k}}_{R}^{2}\\ 
		&+\gamma_x^2\norm{\vI-\vR}_{R}^{2}\norm{\tX^{k}-\widehat{\vX}^{k}}_{R}^{2}\Big),
	\end{align}}\normalsize
	where the  inequality is due to Lemma \ref{lem:UV} with $\tau_1=\theta_{R}\gamma_x$ and $\norm{\Pi_{R}\vR_{\gamma}}_{R}\leq 1-\theta_{R}\gamma_x$. 
	Using Lemma \ref{Norm_R_C}, we obtain
	{\small
	\begin{align}\label{PiRXkNorm}
		\nonumber	&\norm{\Pi_{R}\vX^{k+1}}_{R}^{2}\\
		\nonumber\leq&(1-\theta_{R}\gamma_x)\norm{\Pi_{R}\vX^{k}}_{R}^{2}\\
		\nonumber&+\frac{2}{\theta_{R}\gamma_x}\Big( (1-\theta_{R}\gamma_x)^{2}\delta_{R,2}^2\norm{\Lambda}_{2}^{2}\norm{\vY^{k}}_{F}^{2}\\  &+\delta_{R,2}^2\gamma_x^2\norm{\vI-\vR}_{R}^{2}\norm{\tX^{k}-\widehat{\vX}^{k}}_{F}^{2}\Big).
	\end{align}}\normalsize
	
	Taking the conditional expectation on $\cF^{k}$ yields
	{\small
	\begin{align}\label{PiXkb1}
		\nonumber	&\EE[\norm{\Pi_{R}\vX^{k+1}}_{R}^{2}\vert \cF^{k}]\\
		\nonumber	\leq&(1-\theta_{R}\gamma_x)\norm{\Pi_{R}\vX^{k}}_{R}^{2}\\
		\nonumber &+\frac{2}{\theta_{R}\gamma_x}\Big[ (1-\theta_{R}\gamma_x)^{2}\delta_{R,2}^2\norm{\Lambda}_{2}^{2}\norm{\vY^{k}}_{F}^{2}\\ \nonumber&+\delta_{R,2}^2\gamma_x^2\norm{\vI-\vR}_{R}^{2}C\norm{\tX^{k}-\vH_{x}^{k}}_{F}^{2}\\
		&+\delta_{R,2}^2\gamma_x^2\norm{\vI-\vR}_{R}^{2} s_k^{2}\sigma^2\Big],
	\end{align}}\normalsize
	where we use the fact that
	{\small
	\begin{align}\label{tXk_hXkb}
		\EE[\norm{\tX^{k}-\widehat{\vX}^{k}}_{F}^{2} \vert \cF^{k}]\leq& C \norm{\tX^{k}-\vH_{x}^{k}}_{F}^{2}+s_k^{2}\sigma^2.
	\end{align}}\normalsize
	For convenience, we rewrite \eqref{PiXkb1} as
	{\small
	\begin{align}\label{PiXkb2}
		\nonumber	&\EE[\norm{\Pi_{R}\vX^{k+1}}_{R}^{2}\vert \cF^{k}]\\
		\nonumber \leq&(1-\theta_{R}\gamma_x)\norm{\Pi_{R}\vX^{k}}_{R}^{2}\\
		\nonumber	&+\frac{1}{\theta_{R}\gamma_x}\Big[ c_2(1-\theta_{R}\gamma_x)^{2}\hlambda^{2}\norm{\vY^{k}}_{F}^{2}\\ &+c_{24}C\gamma_x^2\norm{\tX^{k}-\vH_{x}^{k}}_{F}^{2}\Big]+ s_k^{2}\zeta_{c},
	\end{align}}\normalsize
	where 
		{\small
		$$c_2= \delta_{R,2}^2,
		c_{24}=18\delta_{R,2}^2\ge 2\delta_{R,2}^2\norm{\vI-\vR}_{R}^{2},$$}\normalsize
		and 
		{\small
		\begin{equation*}
			\begin{aligned}
				\zeta_{c}=&18\delta_{R,2}^2\sigma^2/\theta_{R}\\
				\geq& 2\sigma^2\delta_{R,2}^2\norm{\vI-\vR}_{R}^{2}/\theta_{R}\\
				\geq& 2\sigma^2\delta_{R,2}^2\norm{\vI-\vR}_{R}^{2}\gamma_x^2/(\theta_{R}\gamma_x) 
			\end{aligned}
		\end{equation*}}\normalsize
		hold from Lemma \ref{Norm_R_C}.

	\subsubsection{Second  inequality}
	By relation \eqref{PiCYk}, we obtain
	{\small
	\begin{align}\label{PiCYkNorm0}
		\nonumber	&\norm{\Pi_{C}\vY^{k+1}}_{C}^{2}\\
		\nonumber	=&\Big\Vert \Pi_{C}\vC_{\gamma}\Pi_{C}\vY^{k}+\Pi_{C}\vC_{\gamma}(\nabla\vF(\vX^{k+1})-\nabla\vF(\vX^{k}))\\
		\nonumber &	+\gamma_{y}(\vI-\vC)(\tY^{k}-\widehat{\vY}^{k})\Big\Vert_{C}^{2}\\
		\nonumber	\leq&(1-\theta_{C}\gamma_y)\norm{\Pi_{C}\vY^{k}}_{C}^{2}\\
		\nonumber &+\frac{2}{\theta_{C}\gamma_y}\Big( (1-\theta_{C}\gamma_y)^{2}\norm{\nabla\vF(\vX^{k+1})-\nabla\vF(\vX^{k})}_{C}^{2}\\ 
		& +\gamma_y^2\norm{(\vI-\vC)(\tY^{k}-\widehat{\vY}^{k})}_{C}^{2}\Big).
	\end{align}}\normalsize
	where Lemma \ref{lem:UV} is used. Recalling Lemma \ref{Norm_R_C}, we have
	{\small
	\begin{align}\label{PiCYkNorm}
		\nonumber	&\norm{\Pi_{C}\vY^{k+1}}_{C}^{2}\\
		\nonumber	\leq&(1-\theta_{C}\gamma_y)\norm{\Pi_{C}\vY^{k}}_{C}^{2}\\
		\nonumber &+\frac{2}{\theta_{C}\gamma_y}\Big( (1-\theta_{C}\gamma_y)^{2}\delta_{C,2}^2\norm{\nabla\vF(\vX^{k+1})-\nabla\vF(\vX^{k})}_{F}^{2}\\ &+\delta_{C,2}^2\gamma_y^2\norm{\vI-\vC}_{C}^{2}\norm{\tY^{k}-\widehat{\vY}^{k}}_{F}^{2}\Big). 
	\end{align}}\normalsize
	Note that 
	\begin{align}\label{tYk_hYk}
		\nonumber	&\EE[\norm{\tY^{k}-\widehat{\vY}^{k}}_{F}^{2} \vert \cF^{k}]\\
		\nonumber\leq& C \EE[\norm{\tY^{k}-\vH_{y}^{k}}_{F}^{2}\vert \cF^{k}]+s_k^{2}\sigma^2\\
		\nonumber \leq& 2  C\norm{\vY^{k}-\vH_{y}^{k}}_{F}^{2}\\
		& +2 C\EE[\norm{\nabla\vF(\vX^{k+1})-\nabla\vF(\vX^{k})}_{F}^{2}\vert \cF^{k}]+s_k^{2}\sigma^2.
	\end{align}
	Taking the conditional expectation on both sides of \eqref{PiCYkNorm}, we obtain
	{\small
	\begin{align}
		\nonumber	&\EE[\norm{\Pi_{C}\vY^{k+1}}_{C}^{2}\vert \cF^{k}]\\
		\nonumber	\leq&(1-\theta_{C}\gamma_y)\norm{\Pi_{C}\vY^{k}}_{C}^{2}\\
		\nonumber &+\frac{2}{\theta_{C}\gamma_y}\Big[ 2C\delta_{C,2}^2\gamma_y^2\norm{\vI-\vC}_{2}^{2}\norm{\vY^{k}-\vH_{x}^{k}}_{F}^{2}
		+\Big((1-\theta_{C}\gamma_y)^{2}\\ 
		\nonumber &+2C\gamma_y^2\norm{\vI-\vC}_{C}^{2}\Big)\delta_{C,2}^2
		\EE[\norm{\nabla\vF(\vX^{k+1})-\nabla\vF(\vX^{k})}_{F}^{2}\vert \cF^{k}]\\ 
		\label{PiCYkNorm1} &
		+\delta_{C,2}^2\gamma_y^2\norm{\vI-\vC}_{C}^{2}s_k^{2}\sigma^2\Big].
	\end{align}}\normalsize
	Recalling that 
	{\small
	\begin{align}\label{nFp_nFb}
		\norm{\nabla\vF(\vX^{k+1})-\nabla\vF(\vX^{k})}_{F}^{2}\leq L^2 \norm{\vX^{k+1}-\vX^{k}}_{F}^{2},
	\end{align}}\normalsize
	we need to bound $\norm{\vX^{k+1}-\vX^{k}}_{F}^{2}$. From the update of the decision variables in Algorithm \ref{Alg:RCPPe}, we know
	{\small
	\begin{align}
		\nonumber	\vX^{k+1}-\vX^{k}
		=&\widetilde{\vX}^{k}-\gamma_{x}(\vI-\vR)\widehat{\vX}^{k}-\vX^{k}\\
		\nonumber=&\gamma_{x}(\vI-\vR)(\widetilde{\vX}^{k}-\widehat{\vX}^{k})-\gamma_{x}(\vI-\vR)\Pi_{R}\vX^{k}\\
		&-\vR_{\gamma}\Lambda \vY^{k}.
	\end{align}}\normalsize
	Thus, we have 
	{\small
	\begin{align}\label{Xkp_XkNorm}	
		\nonumber	&\norm{\vX^{k+1}-\vX^{k}}_{F}^{2}\\
		\nonumber	\leq&3\norm{\gamma_{x}(\vI-\vR)(\widetilde{\vX}^{k}-\widehat{\vX}^{k})}_{F}^{2}
		+3\norm{\gamma_{x}(\vI-\vR)\Pi_{R}\vX^{k}}_{F}^{2}\\
		\nonumber	&+3\norm{\vR_{\gamma}\Lambda \vY^{k}}_{F}^{2}\\
		\nonumber	\leq&3\gamma_{x}^{2}\norm{\vI-\vR}_{R}^{2}\norm{\widetilde{\vX}^{k}-\widehat{\vX}^{k}}_{F}^{2}
		+3\gamma_{x}^{2}\norm{\vI-\vR}_{R}^{2}\norm{\Pi_{R}\vX^{k}}_{R}^{2}\\
		&+3\hlambda^2\norm{\vR_{\gamma}}_{R}^{2}\norm{ \vY^{k}}_{F}^{2}.
	\end{align}}\normalsize
	Putting \eqref{Xkp_XkNorm}  and  \eqref{nFp_nFb}  back into \eqref{PiCYkNorm1} and using relation \eqref{tXk_hXkb}, we have
	{\small
	\begin{align}\label{PiCYkNorm2}
		\nonumber	&\EE[\norm{\Pi_{C}\vY^{k+1}}_{C}^{2}\vert \cF^{k}]\\
		\nonumber	\leq&(1-\theta_{C}\gamma_y)\norm{\Pi_{C}\vY^{k}}_{C}^{2}
		+\frac{2}{\theta_{C}\gamma_y}\Big[ \Big((1-\theta_{C}\gamma_y)^{2}\\
		\nonumber   &+2C\gamma_y^2\norm{\vI-\vC}_{C}^{2}\Big)\delta_{C,2}^2
		L^2\Big(
		3\gamma_{x}^{2}\norm{\vI-\vR}_{R}^{2}C\norm{\widetilde{\vX}^{k}-\vH_{x}^{k}}_{F}^{2}\\
		\nonumber	&+3\gamma_{x}^{2}\norm{\vI-\vR}_{R}^{2}s_k^{2}\sigma^2
		+3\gamma_{x}^{2}\norm{\vI-\vR}_{R}^{2}\norm{\Pi_{R}\vX^{k}}_{F}^{2}\\
		\nonumber&+3\hlambda^2\norm{\vR_{\gamma}}_{R}^{2}\norm{ \vY^{k}}_{F}^{2} \Big)\\
		\nonumber&+2C\delta_{C,2}^2\gamma_y^2\norm{\vI-\vC}_{C}^{2}\norm{\vY^{k}-\vH_{x}^{k}}_{F}^{2}\\
		&+\delta_{C,2}^2\gamma_y^2\norm{\vI-\vC}_{C}^{2}s_k^{2}\sigma^2\Big].
	\end{align}}\normalsize
	For simplicity, we reorganize \eqref{PiCYkNorm2} as
	{\small
	\begin{align}\label{PiCYkNorm3}
		\nonumber	&\EE[\norm{\Pi_{C}\vY^{k+1}}_{C}^{2}\vert \cF^{k}]\\
		\nonumber	\leq&(1-\theta_{C}\gamma_y)\norm{\Pi_{C}\vY^{k}}_{C}^{2}\\
		\nonumber	&+\frac{1}{\theta_{C}\gamma_y}\Big[ e_1
		L^2\Big(
		c_{34}C\gamma_{x}^{2}\norm{\widetilde{\vX}^{k}-\vH_{x}^{k}}_{F}^{2}
		+c_{32}\gamma_{x}^{2}\norm{\Pi_{R}\vX^{k}}_{F}^{2}\\
		&+ c_3\hlambda^2\norm{ \vY^{k}}_{F}^{2}  \Big)
		+c_{35}C \gamma_y^2 \norm{\vY^{k}-\vH_{x}^{k}}_{F}^{2}\Big]
		+s_k^{2}\zeta_{g},
	\end{align}}\normalsize
	where 
		{\small$$c_3=12\geq 3\norm{\vR_{\gamma}}_{R}^{2},
		d_1=2\delta_{C,2}^2,$$
		$$d_2=36\delta_{C,2}^2\geq4\delta_{C,2}^2\norm{\vI-\vC}_{C}^{2},$$ 
		$$e_1=2\delta_{C,2}^2(1+18C)\geq d_1 +d_2C  \geq d_1(1-\theta_{C}\gamma_y)^{2}+d_2C\gamma_y^2,$$ 
		$$c_{32}=c_{34}=27\geq 3\norm{\vI-\vR}_{R}^{2},$$ 
		$$c_{35}=36\delta_{C,2}^2\geq4 \delta_{C,2}^2 \norm{\vI-\vC}_{C}^{2},$$ }\normalsize
		and 
		{\small
		\begin{equation*}
			\begin{aligned}
				\zeta_{g}=&108\delta_{C,2}^2(1+18C)L^2\sigma^2/\theta_{C}+18\delta_{C,2}^2\sigma^2/\theta_{C}\\
				\geq& 6\sigma^2/\theta_{C} \norm{\vI-\vR}_{R}^{2} L^2\big(d_1+d_2C\big) \\
				&+2\sigma^2/\theta_{C}\delta_{C,2}^2\norm{\vI-\vC}_{C}^{2} \\
				\geq& 6\sigma^2 \norm{\vI-\vR}_{R}^{2} L^2\big(d_1(1-\theta_{C}\gamma_y)^{2}+d_2C\gamma_y^2\big)\gamma_{x}^{2}/(\theta_{C}\gamma_y)\\ &+2\sigma^2\delta_{C,2}^2\norm{\vR_{\gamma}}_{R}^{2}\norm{\vI-\vC}_{C}^{2} \gamma_y^2/(\theta_{C}\gamma_y)
			\end{aligned}
		\end{equation*}}\normalsize
		hold due to  $\gamma_x\leq\gamma_y\leq 1$ and Lemma \ref{Norm_R_C}.
	
	\subsubsection{Third  inequality}
	Recalling the  update of the variables $\vH_{x}^{k+1}$, $\tX^{k}$, and $\vY^{k+1}$ in Algorithm \ref{Alg:RCPPe}, we have
	{\small
	\begin{equation*}
		\begin{aligned}
			&\tX^{k+1}-\vH_{x}^{k+1}\\
			=&\tX^{k+1}-\tX^{k}+\tX^{k}-(\vH^{k}_{x}+\alpha_x\vQ^{k}_{x})\\
			=&\tX^{k+1}-\tX^{k}+(1-\alpha_x r)(\tX^{k}-\vH^{k}_{x})\\
			&+\alpha_x r(\tX^{k}-\vH^{k}_{x}- \vQ^{k}_{x}/r)	\\
			=&\vX^{k+1}-\vX^{k}-\Lambda(\vY^{k+1}-\vY^{k})\\
			&+(1-\alpha_x r)(\tX^{k}-\vH^{k}_{x})+\alpha_x r(\tX^{k}-\vH^{k}_{x}- \vQ^{k}_{x}/r)
		\end{aligned}
	\end{equation*}}\normalsize
	and 
	{\small
	\begin{align}\label{Ykp_Yk}
		\nonumber	\vY^{k+1}-\vY^{k}
		=&\tY^{k}-\gamma_{y} (\vI-\vC)\widehat{\vY}^{k}-\vY^{k}\\
		\nonumber=&\gamma_{y}(\vI-\vC)(\tY^{k}-\widehat{\vY}^{k})-\gamma_{y}(\vI-\vC)\Pi_{C}\vY^{k}\\
		&+\vC_{\gamma}(\nabla\vF(\vX^{k+1})-\nabla\vF(\vX^{k})).
	\end{align}}\normalsize
	Based on Lemma \ref{lem:UV},  we get
	{\small
	\begin{align}
		\nonumber	&\norm{\tX^{k+1}-\vH_{x}^{k+1}}_{F}^{2}\\
		\nonumber\leq&(1+\frac{2}{\alpha_x r \delta})\norm{\vX^{k+1}-\vX^{k}-\Lambda(\vY^{k+1}-\vY^{k})}_{F}^{2}\\
		\nonumber&+(1+\frac{\alpha_x r \delta}{2})\Big\Vert(1-\alpha_x r)(\tX^{k}-\vH^{k}_{x})\\
		\nonumber&~+\alpha_x r(\tX^{k}-\vH^{k}_{x}- \vQ^{k}_{x}/r)\Big\Vert_{F}^{2}.
	\end{align}}\normalsize
	Taking the conditional expectation on $\cF^{k}$ yields
	{\small
	\begin{align}\label{tXkp_Hxkp}
		\nonumber&\EE[\norm{\tX^{k+1}-\vH_{x}^{k+1}}_{F}^{2} \vert \cF^{k}]\\
		\nonumber	\leq & (1+\frac{2}{\alpha_x r \delta})\Big(2\EE[\norm{\vX^{k+1}-\vX^{k}}_{F}^{2}\vert \cF^{k}]\\
		\nonumber&+2\EE[\norm{\Lambda}_{2}^{2}\norm{\vY^{k+1}-\vY^{k})}_{F}^{2}\vert \cF^{k}]\Big)\\
		\nonumber	&+(1+\frac{\alpha_x r \delta}{2})(1-\alpha_x r \delta)\norm{\tX^{k}-\vH^{k}_{x}}_{F}^{2}\\
		&+s_k^{2}\alpha_x r \sigma^2_{r}(1+\frac{\alpha_x r \delta}{2}),
	\end{align}}\normalsize
	where we use 
	{\small
	\begin{align*}
		&\EE[\norm{(1-\alpha_x r)(\tX^{k}-\vH^{k}_{x})+\alpha_x r(\tX^{k}-\vH^{k}_{x}- \vQ^{k}_{x}/r)}_{F}^{2}\vert \cF^{k}]\\
		\leq&(1-\alpha_x r)\norm{ \tX^{k}-\vH^{k}_{x} }_{F}^{2}+\alpha_x r \EE[\norm{ \tX^{k}-\vH^{k}_{x}- \vQ^{k}_{x}/r }_{F}^{2}\vert \cF^{k}]\\
		\leq& (1-\alpha_x r)\norm{ \tX^{k}-\vH^{k}_{x} }_{F}^{2}+\alpha_x r(1-\delta)\norm{ \tX^{k}-\vH^{k}_{x} }_{F}^{2} \\
		&+\alpha_x r s_k^2 \sigma^2_{r}\\
		=& (1-\alpha_x r\delta)\norm{ \tX^{k}-\vH^{k}_{x} }_{F}^{2}  +s_k^2\alpha_x r \sigma^2_{r}.
	\end{align*}}\normalsize
	Then, we bound $\EE[\norm{\vY^{k+1}-\vY^{k}}_{F}^{2}\vert \cF^{k}]$.  Repeatedly using Lemma \ref{lem:UV} together with relation \eqref{Ykp_Yk}, we have 
	{\small
	\begin{align}
		\nonumber &\EE[\norm{\vY^{k+1}-\vY^{k}}_{F}^{2}\vert \cF^{k}]\\
		\nonumber\leq&3 \EE[\norm{ \gamma_{y}(\vI-\vC)(\tY^{k}-\widehat{\vY}^{k})}_{F}^{2}\vert \cF^{k}]
		+3 \norm{\gamma_{y}(\vI-\vC)\Pi_{C}\vY^{k} }_{F}^{2}\\
		\nonumber&+3 \EE[\norm{ \vC_{\gamma}(\nabla\vF(\vX^{k+1})-\nabla\vF(\vX^{k}))}_{F}^{2}\vert \cF^{k}]\\
		\nonumber\leq&3\gamma_{y}^2\norm{ \vI-\vC}_{C}^{2}\EE[\norm{\tY^{k}-\widehat{\vY}^{k}}_{F}^{2}\vert \cF^{k}]\\
		\nonumber&+3\gamma_{y}^2\norm{ \vI-\vC}_{C}^{2} \norm{ \Pi_{C}\vY^{k} }_{F}^{2}\\
		\nonumber&+3\norm{ \vC_{\gamma} }_{C}^{2} \EE[\norm{  \nabla\vF(\vX^{k+1})-\nabla\vF(\vX^{k})}_{F}^{2}\vert \cF^{k}].
	\end{align}}\normalsize 
	Using \eqref{tYk_hYk}, we obtain
	{\small
	\begin{align}\label{Ykp_Yk1_0}
		\nonumber &\EE[\norm{\vY^{k+1}-\vY^{k}}_{F}^{2}\vert \cF^{k}]\\
		\nonumber \leq&3\gamma_{y}^2\norm{ \vI-\vC}_{2}^{2}\Big( 2C\norm{\vY^{k}-\vH_{y}^{k}}_{F}^{2}\\ \nonumber&+2C\EE[\norm{  \nabla\vF(\vX^{k+1})-\nabla\vF(\vX^{k})}_{F}^{2}\vert \cF^{k}]+s_{k}^2\sigma^2\Big)\\
		\nonumber&
		+3\gamma_{y}^2\norm{ \vI-\vC}_{C}^{2} \norm{ \Pi_{C}\vY^{k} }_{F}^{2}\\
		&+3\norm{ \vC_{\gamma} }_{C}^{2} \EE[\norm{  \nabla\vF(\vX^{k+1})-\nabla\vF(\vX^{k})}_{F}^{2}\vert \cF^{k}].
	\end{align}}\normalsize
	Putting the relation \eqref{nFp_nFb} into \eqref{Ykp_Yk1_0}, we know
	{\small
	\begin{align}\label{Ykp_Yk1}
		\nonumber &\EE[\norm{\vY^{k+1}-\vY^{k}}_{F}^{2}\vert \cF^{k}]\\
		\nonumber\leq&6C\gamma_{y}^2\norm{ \vI-\vC}_{C}^{2}  \norm{\vY^{k}-\vH_{y}^{k}}_{F}^{2} 
		+3\gamma_{y}^2\norm{ \vI-\vC}_{C}^{2}s_{k}^2\sigma^2\\
		\nonumber&+3\gamma_{y}^2\norm{ \vI-\vC}_{C}^{2} \norm{ \Pi_{C}\vY^{k} }_{F}^{2}
		+\Big(3\norm{ \vC_{\gamma} }_{C}^{2}\\
		&~+6C\gamma_{y}^2\norm{ \vI-\vC}_{C}^{2}\Big)L^2 \EE[\norm{\vX^{k+1}-\vX^{k}}_{F}^{2}\vert \cF^{k}].
	\end{align}}\normalsize
	Substituting \eqref{Ykp_Yk1} into \eqref{tXkp_Hxkp}, we have
	{\small
	\begin{align}
		\nonumber&\EE[\norm{\tX^{k+1}-\vH_{x}^{k+1}}_{F}^{2} \vert \cF^{k}]\\
		\nonumber \leq& (1+\frac{2}{\alpha_x r \delta})\Big[\Big(2+6L^2\hlambda^2 \norm{\vC_{\gamma}}_{C}^{2}\\
		\nonumber&~+12CL^2\hlambda^2  \gamma_{y}^2 \norm{\vI-\vC}_{C}^{2}\Big)\EE[\norm{\vX^{k+1}-\vX^{k}}_{F}^{2}\vert \cF^{k}]\\
		\nonumber	&+12C\hlambda^2  \gamma_{y}^2 \norm{\vI-\vC}_{2}^{2} \norm{ \vY^{k}-\vH_{y}^{k}}_{F}^{2}\\
		\nonumber	&+6\hlambda^2 \gamma_{y}^2 \norm{\vI-\vC}_{2}^{2}\norm{\Pi_{C}\vY^{k} }_{C}^{2}
		\Big]\\
		\nonumber	&+(1-\frac{\alpha_x r \delta}{2})\norm{\tX^{k}-\vH^{k}_{x}}_{F}^{2}+s_k^{2}\alpha_x r \sigma^2_{r}(1+\frac{\alpha_x r \delta}{2})\\
		&+6s_{k}^2\sigma^2\hlambda^2  \gamma_{y}^2 \norm{\vI-\vC}_{C}^{2} (1+\frac{2}{\alpha_x r \delta}).
	\end{align}}\normalsize
	
	Noting that  $6L^2\hlambda^2 \norm{\vC_{\gamma}}_{C}^{2}\leq 1 $, $12CL^2\hlambda^2   \norm{\vI-\vC}_{C}^{2}\leq 1 $
	and $\gamma_{y}^2\leq 1$ from the assumption, we have  $2+6L^2\hlambda^2 \norm{\vC_{\gamma}}_{C}^{2} +12CL^2\hlambda^2  \norm{\vI-\vC}_{C}^{2}\gamma_{y}^2\leq 4$. 
	Based on $1+\frac{2}{\alpha_x r \delta}\leq \frac{3}{\alpha_x r \delta}$ and $1+\frac{\alpha_x r \delta}{2}\leq 2$, we obtain
	{\small
	\begin{align}\label{tXkp_Hxkp2}
		\nonumber&\EE[\norm{\tX^{k+1}-\vH_{x}^{k+1}}_{F}^{2} \vert \cF^{k}]\\
		\nonumber\leq& \frac{3}{\alpha_x r \delta}\Big(4\EE[\norm{\vX^{k+1}-\vX^{k}}_{F}^{2}\vert \cF^{k}]\\
		\nonumber&+12C\hlambda^2 \gamma_{y}^2   \norm{\vI-\vC}_{C}^{2} \norm{ \vY^{k}-\vH_{y}^{k}}_{F}^{2}\\
		\nonumber&+6\hlambda^2  \gamma_{y}^2 \norm{\vI-\vC}_{C}^{2}\norm{\Pi_{C}\vY^{k} }_{C}^{2}
		\Big)\\
		\nonumber&+(1-\frac{\alpha_x r \delta}{2})\norm{\tX^{k}-\vH^{k}_{x}}_{F}^{2}\\
		&+2s_k^{2}\alpha_x r \sigma^2_{r}
		+s_{k}^2\sigma^2\hlambda^2  \gamma_{y}^2 \norm{\vI-\vC}_{C}^{2} \frac{18}{\alpha_x r \delta}.
	\end{align}}\normalsize
	Plugging \eqref{Xkp_XkNorm} into \eqref{tXkp_Hxkp2}, we get
	{\small
	\begin{align}
		\nonumber&\EE[\norm{\tX^{k+1}-\vH_{x}^{k+1}}_{F}^{2} \vert \cF^{k}]\\
		\nonumber\leq& \frac{3}{\alpha_x r \delta}\Big(12 \gamma_{x}^{2}\norm{\vI-\vR}_{R}^{2}C\norm{\widetilde{\vX}^{k}-\vH_{x}^{k}}_{F}^{2} \\
		\nonumber&+12\gamma_{x}^{2}\norm{\vI-\vR}_{R}^{2}s_k^{2}\sigma^2
		+12\gamma_{x}^{2}\norm{\vI-\vR}_{R}^{2}\norm{\Pi_{R}\vX^{k}}_{R}^{2}\\
		\nonumber&+12\hlambda^2\norm{\vR_{\gamma}}_{R}^{2}\norm{ \vY^{k}}_{F}^{2}
		+12C\hlambda^2 \gamma_{y}^2   \norm{\vI-\vC}_{C}^{2} \norm{ \vY^{k}-\vH_{y}^{k}}_{F}^{2}\\
		\nonumber&+6\hlambda^2 \gamma_{y}^2  \norm{\vI-\vC}_{C}^{2}\norm{\Pi_{C}\vY^{k} }_{C}^{2}
		\Big)\\
		\nonumber&+(1-\frac{\alpha_x r \delta}{2})\norm{\tX^{k}-\vH^{k}_{x}}_{F}^{2}\\
		&+2s_k^{2}\alpha_x r \sigma^2_{r}
		+s_{k}^2\sigma^2 \hlambda^2  \gamma_{y}^2 \norm{\vI-\vC}_{C}^{2}   \frac{18}{\alpha_x r \delta}.\label{tXkp_Hxkp3}
	\end{align}}\normalsize
	For convenience, we simplify it as
	{\small
	\begin{align}
		\nonumber&\EE[\norm{\tX^{k+1}-\vH_{x}^{k+1}}_{F}^{2} \vert \cF^{k}]\\
		\nonumber\leq& \frac{1}{\alpha_x r \delta}\Big(d_{44} C\gamma_{x}^{2}\norm{\widetilde{\vX}^{k}-\vH_{x}^{k}}_{F}^{2} 
		+d_{42}\gamma_{x}^{2}\norm{\Pi_{R}\vX^{k}}_{R}^{2}\\
		\nonumber&+ c_4\hlambda^2\norm{ \vY^{k}}_{F}^{2}
		+ c_{43}\hlambda^2\gamma_{y}^2\norm{\Pi_{C}\vY^{k}}_{C}^{2}\\
		\nonumber&+c_{45}C\hlambda^2  \gamma_{y}^2  \norm{ \vY^{k}-\vH_{y}^{k}}_{F}^{2}
		\Big)\\
		\label{tXkp_Hxkp4}&+(1-\frac{\alpha_x r \delta}{2})\norm{\tX^{k}-\vH^{k}_{x}}_{F}^{2}+s_k^{2}\zeta_{cx}
		,
	\end{align}}\normalsize
	where 
		{\small
		$$c_4= 48\geq 12\norm{\vR_{\gamma}}_{R}^{2}, 
		d_{42}=d_{44}=324\geq 36\norm{\vI-\vR}_{R}^{2},$$ 
		$$c_{43}= 162\geq18\norm{\vI-\vC}_{C}^{2},   
		c_{45}=324\geq 36\norm{\vI-\vC}_{C}^{2},$$ 
		\begin{equation*}
			\begin{aligned}
				\zeta_{cx}=&2\alpha_x r \sigma^2_{r}+\sigma^2\Big(  \frac{1}{4L^2}+ 18 \Big) \frac{18}{\alpha_x r \delta}\\
				\geq& 2\alpha_x r \sigma^2_{r}+\sigma^2\Big(  9\hlambda^2+ 18 \Big) \frac{18}{\alpha_x r \delta}\\
				\geq& 2\alpha_x r \sigma^2_{r} +\sigma^2\Big(  \norm{\vI-\vC}_{C}^{2}\hlambda^2   +  2\norm{\vI-\vR}_{R}^{2} \Big) \frac{18}{\alpha_x r \delta}\\
				\geq& 2\alpha_x r \sigma^2_{r}+\sigma^2\Big(  \norm{\vI-\vC}_{C}^{2}\hlambda^2  \gamma_{y}^2 +  2\norm{\vI-\vR}_{R}^{2}\gamma_{x}^{2} \Big)  \frac{18}{\alpha_x r \delta} 
			\end{aligned}
		\end{equation*}}\normalsize
		hold from Lemma \ref{Norm_R_C} and the assumption $\hlambda\leq \frac{1}{6L}$.

	\subsubsection{Fourth  inequality}
	Recalling the  updates of the variables $\vH_{y}^{k+1}$ and $\vY^{k+1}$ in Algorithm \ref{Alg:RCPPe}, we know
	{\small
	\begin{align}
		\nonumber	&\vY^{k+1}-\vH_{y}^{k+1}\\
		\nonumber=&\vY^{k+1}-\tY^{k}+\tY^{k}-(\vH^{k}_{y}+\alpha_y\vQ^{k}_{y})\\
		\nonumber	=&\vY^{k+1}-\tY^{k}+(1-\alpha_y r)(\tY^{k}-\vH^{k}_{y})\\
		\nonumber	& +\alpha_y r(\tY^{k}-\vH^{k}_{y}- \vQ^{k}_{y}/r).
	\end{align}}\normalsize
	Based on Lemma \ref{lem:UV},  we have
	{\small
	\begin{align}\label{Ykp_Hkyp1}
		\nonumber&\EE[\norm{ \vY^{k+1}-\vH_{y}^{k+1}}_{F}^{2} \vert \cF^{k}]\\
		\nonumber\leq& (1+\frac{2}{\alpha_y r \delta})\EE[\norm{\vY^{k+1}-\tY^{k} }_{F}^{2}\vert \cF^{k}]\\
		&+ (1-\frac{\alpha_y r \delta}{2})\EE[\norm{\tY^{k}-\vH^{k}_{y}}_{F}^{2}\vert \cF^{k}]
		+2s_k^{2}\alpha_y r \sigma^2_{r},
	\end{align}}\normalsize
	where similar technique as in \eqref{tXkp_Hxkp} is used. The following step is to bound $\EE[\norm{\vY^{k+1}-\tY^{k} }_{F}^{2}\vert \cF^{k}]$. Recalling the update of $\vY^{k+1}$ in Algorithm \ref{Alg:RCPPe}, we get
	{\small
	\begin{equation*}
		\vY^{k+1}-\tY^{k}=\gamma_{y}(\vI-\vC)(\tY^{k}-\widehat{\vY}^{k})-\gamma_{y}(\vI-\vC)\tY^{k}.
	\end{equation*}}\normalsize
	Using Lemma \ref{lem:UV}, we have
	{\small
	\begin{align*}
		&\EE[\norm{\vY^{k+1}-\tY^{k} }_{F}^{2}\vert \cF^{k}]\\
		\nonumber \leq& 2\EE[\norm{\gamma_{y}(\vI-\vC)(\tY^{k}-\widehat{\vY}^{k})}_{F}^{2}\vert \cF^{k}]\\
		\nonumber &+2\EE[\norm{ \gamma_{y}(\vI-\vC)\tY^{k} }_{F}^{2}\vert \cF^{k}].
	\end{align*}}\normalsize
	Reviewing the update $\tY^{k}=\vY^{k}+\nabla\vF(\vX^{k+1})-\nabla\vF(\vX^{k})$, we know
	{\small
	\begin{align*}
		&\EE[\norm{\vY^{k+1}-\tY^{k} }_{F}^{2}\vert \cF^{k}]\\
		\leq& 2\EE[\norm{\gamma_{y}(\vI-\vC)(\tY^{k}-\widehat{\vY}^{k})}_{F}^{2}\vert \cF^{k}]\\
		&+4\EE[\norm{ \gamma_{y}(\vI-\vC)\Pi_{C}\vY^{k}}_{F}^{2}\vert \cF^{k}]\\
		&+4\EE[\norm{ \gamma_{y}(\vI-\vC)(\nabla\vF(\vX^{k+1})-\nabla\vF(\vX^{k})) }_{F}^{2}\vert \cF^{k}].
	\end{align*}}\normalsize
	Using \eqref{tYk_hYk}, we further have
	{\small
	\begin{align}\label{Ykp_tYk1}
		\nonumber &\EE[\norm{\vY^{k+1}-\tY^{k} }_{F}^{2}\vert \cF^{k}]\\
		\nonumber \leq& 4\gamma_{y}^2\norm{ \vI-\vC}_{C}^{2}C \norm{ \vY^{k}-\vH^{k}_{y} }_{F}^{2}\\
		\nonumber &+4\gamma_{y}^2\norm{ \vI-\vC}_{C}^{2}C\EE[\norm{ \nabla\vF(\vX^{k+1})-\nabla\vF(\vX^{k})) }_{F}^{2}\vert \cF^{k}]\\
		\nonumber &+2\gamma_{y}^2\norm{ \vI-\vC}_{C}^{2}s_{k}^2\sigma^2
		+4\gamma_{y}^2\norm{ \vI-\vC}_{C}^{2} \norm{ \Pi_{C}\vY^{k} }_{F}^{2} \\
		\nonumber &+4\gamma_{y}^2\norm{ \vI-\vC}_{C}^{2} \EE[\norm{  \nabla\vF(\vX^{k+1})-\nabla\vF(\vX^{k}))  }_{F}^{2}\vert \cF^{k}]\\
		\nonumber \leq& 4\gamma_{y}^2\norm{ \vI-\vC}_{C}^{2}C \norm{ \vY^{k}-\vH^{k}_{y} }_{F}^{2}
		+4\gamma_{y}^2\norm{ \vI-\vC}_{C}^{2} \norm{ \Pi_{C}\vY^{k} }_{F}^{2}\\
		\nonumber&+4(C+1)\gamma_{y}^2\norm{ \vI-\vC}_{C}^{2}L^2\EE[\norm{  \vX^{k+1} -\vX^{k}  }_{F}^{2}\vert \cF^{k}]\\
		&+2\gamma_{y}^2\norm{ \vI-\vC}_{C}^{2}s_{k}^2\sigma^2.
	\end{align}}\normalsize
	Besides, we bound $\EE[\norm{\tY^{k}-\vH^{k}_{y}}_{F}^{2}\vert \cF^{k}]$ as follows.
	{\small
	\begin{align}\label{tYk_Hky1}
		\nonumber	&\EE[\norm{\tY^{k}-\vH^{k}_{y}}_{F}^{2}\vert \cF^{k}]\\
		\nonumber\leq& (1+\frac{\alpha_y r \delta}{4})\norm{\vY^{k}-\vH^{k}_{y}}_{F}^{2}\\
		\nonumber &+(1+\frac{4}{\alpha_y r \delta}) \EE[\norm{ \nabla\vF(\vX^{k+1})-\nabla\vF(\vX^{k})) }_{F}^{2}\vert \cF^{k}]\\
		\nonumber	\leq& (1+\frac{\alpha_y r \delta}{4})\norm{\vY^{k}-\vH^{k}_{y}}_{F}^{2}\\
		&+(1+\frac{4}{\alpha_y r \delta})L^2\EE[\norm{\vX^{k+1}-\vX^{k} }_{F}^{2}\vert \cF^{k}].
	\end{align}}\normalsize
	Putting \eqref{Ykp_tYk1} and \eqref{tYk_Hky1} back into \eqref{Ykp_Hkyp1}, we get
	{\small
	\begin{align}\label{Ykp_Hkyp2_0}
		\nonumber&\EE[\norm{ \vY^{k+1}-\vH_{y}^{k+1}}_{F}^{2} \vert \cF^{k}]\\
		\nonumber \leq& (1+\frac{2}{\alpha_y r \delta})\Big(4\gamma_{y}^2\norm{ \vI-\vC}_{C}^{2}C \norm{ \vY^{k}-\vH^{k}_{y} }_{F}^{2}\\
		\nonumber&+2\gamma_{y}^2\norm{ \vI-\vC}_{C}^{2}s_{k}^2\sigma^2
		+4\gamma_{y}^2\norm{ \vI-\vC}_{C}^{2} \norm{ \Pi_{C}\vY^{k} }_{F}^{2} \\
		\nonumber&+4(C+1)\gamma_{y}^2\norm{ \vI-\vC}_{C}^{2}L^2\EE[\norm{  \vX^{k+1} -\vX^{k}  }_{F}^{2}\vert \cF^{k}]
		\Big)\\
		\nonumber	&+ (1-\frac{\alpha_y r \delta}{2})\Big[(1+\frac{\alpha_y r \delta}{4})\norm{\vY^{k}-\vH^{k}_{y}}_{F}^{2}\\
		\nonumber&+(1+\frac{4}{\alpha_y r \delta})L^2\EE[\norm{\vX^{k+1}-\vX^{k} }_{F}^{2}\vert \cF^{k}]\Big]\\
		&+2s_k^{2}\alpha_y r \sigma^2_{r}.
	\end{align}}\normalsize
	After rearranging \eqref{Ykp_Hkyp2_0}, we obtain
	{\small
	\begin{align}\label{Ykp_Hkyp2}
		\nonumber&\EE[\norm{ \vY^{k+1}-\vH_{y}^{k+1}}_{F}^{2} \vert \cF^{k}]\\
		\nonumber\leq& \left[(1-\frac{\alpha_y r \delta}{4})+ 4C\gamma_{y}^2\norm{\vI-\vC}_{C}^{2} \frac{3}{\alpha_y r \delta}\right]\norm{\vY^{k}-\vH^{k}_{y}}_{F}^{2}\\
		\nonumber&+\Big[4(C+1)L^2\gamma_{y}^2\norm{\vI-\vC}_{C}^{2} \frac{3}{\alpha_y r \delta}+\frac{4}{\alpha_y r \delta}\Big]\\
		\nonumber&~\cdot L^2\EE[\norm{\vX^{k+1}-\vX^{k} }_{F}^{2}\vert \cF^{k}]\\
		\nonumber&+4\gamma_{y}^2\norm{\vI-\vC}_{C}^{2}\frac{3}{\alpha_y r \delta}\norm{\Pi_{C}\vY^{k} }_{C}^{2}\\
		&+2s_{k}^2\sigma^2\gamma_{y}^2\norm{\vI-\vC}_{C}^{2}\frac{3}{\alpha_y r \delta}
		+2s_k^{2}\alpha_y r \sigma^2_{r},
	\end{align}}\normalsize
	where we use the inequalities $1+\frac{2}{\alpha_y r \delta}\leq  \frac{3}{\alpha_y r \delta}$ and $(1-\frac{\alpha_y r \delta}{2})(1+\frac{4}{\alpha_y r \delta})\leq \frac{4}{\alpha_y r \delta}$. 
	Plugging \eqref{Xkp_XkNorm} into \eqref{Ykp_Hkyp2}, we obtain
	{\small
	\begin{align}\label{Ykp_Hkyp3}
		\nonumber&\EE[\norm{ \vY^{k+1}-\vH_{y}^{k+1}}_{F}^{2} \vert \cF^{k}]\\
		\nonumber\leq& \Big[(1-\frac{\alpha_y r \delta}{4})+ 4C\gamma_{y}^2\norm{\vI-\vC}_{C}^{2} \frac{3}{\alpha_y r \delta}\Big]\norm{\vY^{k}-\vH^{k}_{y}}_{F}^{2}\\
		\nonumber&+\Big[4(C+1)L^2\gamma_{y}^2\norm{\vI-\vC}_{C}^{2} \frac{3}{\alpha_y r \delta}+ \frac{4}{\alpha_y r \delta}\Big]\\
		\nonumber&~ \cdot L^2\Big(
		3\gamma_{x}^{2}\norm{\vI-\vR}_{R}^{2}C\norm{\widetilde{\vX}^{k}-\vH_{x}^{k}}_{F}^{2}+3\gamma_{x}^{2}\norm{\vI-\vR}_{R}^{2}s_k^{2}\sigma^2\\
		\nonumber	&
		+3\gamma_{x}^{2}\norm{\vI-\vR}_{R}^{2}\norm{\Pi_{R}\vX^{k}}_{F}^{2}
		+3\hlambda^2\norm{\vR_{\gamma}}_{R}^{2}\norm{ \vY^{k}}_{F}^{2} \Big)\\
		\nonumber&+4\gamma_{y}^2\norm{\vI-\vC}_{C}^{2} \frac{3}{\alpha_y r \delta}\norm{\Pi_{C}\vY^{k} }_{C}^{2}\\
		&+2s_{k}^2\sigma^2\gamma_{y}^2\norm{\vI-\vC}_{C}^{2}\frac{3}{\alpha_y r \delta}
		+2s_k^{2}\alpha_y r \sigma^2_{r}.
	\end{align}}\normalsize
	Then, we get
	{\small
	\begin{align}
		\nonumber&\EE[\norm{ \vY^{k+1}-\vH_{y}^{k+1}}_{F}^{2} \vert \cF^{k}]\\
		\nonumber\leq& (1-\frac{\alpha_y r \delta}{4})\norm{\vY^{k}-\vH^{k}_{y}}_{F}^{2}\\
		\nonumber&+\frac{1}{\alpha_y r \delta}\Big\{
		d_{55}C\gamma_{y}^2 \norm{\vY^{k}-\vH^{k}_{y}}_{F}^{2}
		+d_{53}\gamma_{y}^2 \norm{\Pi_{C}\vY^{k} }_{C}^{2}\\
		\nonumber&+e_2 L^2\big[
		d_{54}C\gamma_{x}^{2} \norm{\widetilde{\vX}^{k}-\vH_{x}^{k}}_{F}^{2}
		+d_{52}\gamma_{x}^{2} \norm{\Pi_{R}\vX^{k}}_{F}^{2}\\
		&~+c_5\hlambda^2\norm{ \vY^{k}}_{F}^{2} \big]
		\Big\}	
		+s_{k}^2\zeta_{cy},
		\label{Ykp_Hkyp4}
	\end{align}}\normalsize
	where 
		{\small
		$$c_5=12\geq3\norm{\vR_{\gamma}}_{R}^{2}, 
		d_3=d_{53}=d_{55}=108\geq 12\norm{\vI-\vC}_{C}^{2},$$  
		$$d_4=4,d_{52}=d_{54}=27\geq  3 \norm{\vI-\vR}_{R}^{2},$$  
		$$e_2=108C+112
		\geq d_3(C+1) + d_4 \geq  d_3(C+1)\gamma_{y}^2+ d_4 ,$$  
	}\normalsize
	and
		{\small
		\begin{equation*}
			\begin{aligned}
				\zeta_{cy}=&\sigma^2\big(18+9e_2L^2\big)\frac{3}{\alpha_y r \delta}
				+ 2\alpha_y r \sigma^2_{r}\\
				\geq& \sigma^2\big(2\norm{\vI-\vC}_{C}^{2}
				+e_2L^2\norm{\vI-\vR}_{R}^{2}\big) \frac{3}{\alpha_y r \delta}
				+ 2\alpha_y r \sigma^2_{r}\\
				\geq& \sigma^2\big(2\norm{\vI-\vC}_{C}^{2}\gamma_{y}^2+e_2 L^2\norm{\vI-\vR}_{R}^{2}\gamma_{x}^{2}\big) \frac{3}{\alpha_y r \delta}+ 
				2\alpha_y r \sigma^2_{r}
			\end{aligned}
		\end{equation*}}\normalsize
		hold from Lemma \ref{Norm_R_C} and $\gamma_{x}\leq\gamma_{y}\leq 1$.

	\subsection{Proof of Lemma \ref{Lem:descent}}\label{Pf:descent}
	From Assumption \ref{Assumption: Lipschitz}, the gradient of $f$ is $L$-Lipschitz continuous. In addition, we know $\oX^{k+1}=\oX^{k}-\frac{1}{n}\vu_{R}^{\T}\Lambda \vY^{k}$. Then, we have
	{\small
		\begin{align}
			\nonumber &f(\oX^{k+1})\\
			\nonumber \leq& f(\oX^{k})+\langle \nabla f(\oX^{k}), \oX^{k+1}-\oX^{k}\rangle 
			+\frac{L}{2}\norm{\oX^{k+1}-\oX^{k}}^2\\
			\nonumber =&f(\oX^{k})-\overline{\lambda}\langle \nabla f(\oX^{k}),  \frac{1}{n\overline{\lambda}}\vu_{R}^{\T}\Lambda \vY^{k}\rangle +\frac{L\overline{\lambda}^2}{2}\norm{\frac{1}{n\overline{\lambda}}\vu_{R}^{\T}\Lambda \vY^{k}}^2.
	\end{align}}\normalsize
	Based on the relation $\langle \va,\vb \rangle=\frac{1}{2}(\norm{\va}^2+\norm{\vb}^2-\norm{\va-\vb}^2)$, we obtain
	{\small	
	\begin{align}
			\nonumber &f(\oX^{k+1})\\
			\nonumber \leq&f(\oX^{k})-\frac{\overline{\lambda}}{2} \Big( \norm{\nabla f(\oX^{k})}^2  +\norm{\frac{1}{n\overline{\lambda}}\vu_{R}^{\T}\Lambda \vY^{k}}^2 \\
			\nonumber &-\norm{\nabla f(\oX^{k})-\frac{1}{n\overline{\lambda}}\vu_{R}^{\T}\Lambda \vY^{k} }^2 \Big) 
			\nonumber+\frac{L\overline{\lambda}^2}{2}\norm{\frac{1}{n\overline{\lambda}}\vu_{R}^{\T}\Lambda \vY^{k}}^2.
	\end{align}}\normalsize
	Note that $\overline{\lambda}\leq\frac{1}{L}$ from the assumption in Lemma \ref{Lem:descent}, we know 
	$\frac{L\overline{\lambda}^2}{2}\norm{\frac{1}{n\overline{\lambda}}\vu_{R}^{\T}\Lambda \vY^{k}}^2
	\leq  \frac{\overline{\lambda}}{2}\norm{\frac{1}{n\overline{\lambda}}\vu_{R}^{\T}\Lambda \vY^{k}}^2$. 
	Meanwhile, we know 
	\begin{align*}
	&\nabla f(\oX^{k})-\frac{1}{n\overline{\lambda}}\vu_{R}^{\T}\Lambda \vY^{k}\\
	=&	\nabla f(\oX^{k})-\oY^{k} -\frac{1}{n\overline{\lambda}}\vu_{R}^{\T}\Lambda \Pi_{C}\vY^{k}.
	\end{align*}
	Thus, we have
	{\small
	\begin{align}
			\nonumber  &f(\oX^{k+1})\\
			\nonumber\leq &f(\oX^{k})-\frac{\overline{\lambda}}{2}\norm{\nabla f(\oX^{k})}^2  \\  
			\nonumber &+\frac{\overline{\lambda}}{2}\norm{\nabla f(\oX^{k})-\oY^{k} -\frac{1}{n\overline{\lambda}}\vu_{R}^{\T}\Lambda \Pi_{C}\vY^{k} }^2 \\
			\nonumber \leq&f(\oX^{k})-\frac{\overline{\lambda}}{2}   \norm{\nabla f(\oX^{k})}^2   +\frac{L^2\overline{\lambda}}{n}\norm{\Pi_{R}\vX^{k}}_{R}^{2}\\
			\nonumber &+\frac{\norm{\vu_{R}}^{2}\hlambda^2}{n^2\overline{\lambda}}\norm{\Pi_{C}\vY^{k}}_{C}^{2}\\
			\nonumber \leq&f(\oX^{k})-\frac{M\hlambda}{2}   \norm{\nabla f(\oX^{k})}^2  \\ 
			\nonumber&+\frac{\norm{\vu_{R}} \norm{\vu_{C}}}{n^2} \hlambda L^2\norm{\Pi_{R}\vX^{k}}_{R}^{2}
			+\frac{\norm{\vu_{R}}^{2}}{n^2M}\hlambda\norm{\Pi_{C}\vY^{k}}_{C}^{2}, 
	\end{align}}\normalsize
	where we use Lemma \ref{lem:UV} in the  second inequality, the assumption $\overline{\lambda}\geq M \hlambda$ and the definition $\overline{\lambda}=\frac{1}{n}\vu_{R}^{\T}\Lambda \vu_{C}\leq \frac{1}{n}\vu_{R}^{\T}\vu_{C}\hlambda\leq \frac{1}{n}\norm{\vu_{R}}\norm{\vu_{C}}\hlambda$ in the third inequality.
	\subsection{Proof of Lemma \ref{Lem:Key}}\label{Pf:Key}
	Let $V^{k}=L^2\Omega_{c}^{k}+A\Omega_{g}^{k}+B\Omega_{cx}^{k}+D\Omega_{cy}^{k}$,
	where 
	$A=\frac{\theta_{C}\gamma_{y}\theta_{R}}{108 e_1\gamma_{x}}$, 
	$B=\frac{L^2 \alpha_x r \delta \theta_{R}}{1296  \gamma_x}$,
	$D=\frac{\theta_{C}\gamma_{y} \alpha_y r \delta \theta_{R}}{108 e_2 \gamma_x}\leq\frac{  \alpha_y r \delta \theta_{R}}{108 e_2  \gamma_x}$. 
	Combining Lemmas \ref{Lem:Yk} and \ref{Lem:RCPP} with the conditions on $\gamma_{x},\gamma_{y},\hlambda$, we have 
	{\small
	\begin{equation}
			\begin{aligned}
				&V^{k+1}\\
				\leq& \Big( 1-\theta_{R}\gamma_x  
				+\frac{A}{\theta_{C}\gamma_y}  e_1 c_{32}\gamma_{x}^2
				+\frac{B/L^2}{\alpha_x r \delta}d_{42}\gamma_{x}^2\\
				&\quad+\frac{D}{\alpha_y r \delta}e_2 d_{52}\gamma_{x}^2
				\Big) L^2 \norm{\Pi_{R}\vX^{k}}_{R}^{2}\\
				&+\Big( 1-\theta_{C}\gamma_y  
				+ \frac{B/L^2}{A\alpha_x r \delta}c_{43}L^2\hlambda^2 \gamma_{y}^2\\
				&\quad+ \frac{D}{A\alpha_y r \delta}d_{53}\gamma_{y}^2
				\Big) A\norm{\Pi_{C}\vY^{k}}_{C}^{2}\\
				&+\Big(1-\frac{\alpha_x r \delta}{2}
				+\frac{L^2}{B\theta_{R}\gamma_x} C c_{24}\gamma_{x}^2
				+ \frac{AL^2}{B\theta_{C}\gamma_y}  e_1 Cc_{34}\gamma_{x}^2\\
				&\quad+\frac{1}{\alpha_x r \delta}Cd_{44}\gamma_{x}^2
				+\frac{DL^2}{B\alpha_y r \delta}e_2 Cd_{54}\gamma_{x}^2
				\Big)B\norm{\tX^{k}-\vH^{k}_{x}}_{F}^{2}\\
				&+\Big(1-\frac{\alpha_y r \delta}{4}
				+\frac{A}{D\theta_{C}\gamma_y}Cc_{35}\gamma_{y}^2
				+\frac{B/L^2}{D\alpha_x r \delta}Cc_{45}L^2\hlambda^2 \gamma_{y}^2\\
				&\quad+\frac{1}{\alpha_y r \delta} Cd_{55}\gamma_{y}^2
				\Big)D\norm{\vY^{k}-\vH^{k}_{y}}_{F}^{2}\\
				&+\Big(
				\frac{1}{\theta_{R}\gamma_x}(1-\theta_{R}\gamma_x)^2c_2L^2\hlambda^2
				+\frac{A}{ \theta_{C}\gamma_y}  e_1 c_3L^2\hlambda^2\\
				&\quad+\frac{B/L^2}{\alpha_x r \delta}c_4L^2\hlambda^2
				+\frac{D}{\alpha_y r \delta}e_2 c_5L^2\hlambda^2
				\Big) \norm{ \vY^{k}}_{F}^{2} +s_k^{2}\zeta_0,
			\end{aligned}\label{Vkplus}
	\end{equation}}\normalsize
	where $\zeta_0=L^2\zeta_{c}+A\zeta_{g}+B\zeta_{cx}+D\zeta_{cy}$. 
	Plugging \eqref{YkNorm} into \eqref{Vkplus}, we have 
	{\small
		\begin{equation}
			\begin{aligned}
				&	V^{k+1}\\
				\leq& \Big( 1-\theta_{R}\gamma_x  
				+\frac{A}{\theta_{C}\gamma_y}  e_1 c_{32}\gamma_{x}^2
				+\frac{B/L^2}{\alpha_x r \delta}d_{42}\gamma_{x}^2\\
				&\quad+\frac{D}{\alpha_y r \delta}e_2 d_{52}\gamma_{x}^2
				\Big) L^2 \norm{\Pi_{R}\vX^{k}}_{R}^{2}\\
				&+\Big( 1-\theta_{C}\gamma_y  
				+ \frac{B/L^2}{A\alpha_x r \delta}c_{43}L^2\hlambda^2\gamma_{y}^2\\
				&\quad+ \frac{D}{A\alpha_y r \delta}d_{53}\gamma_{y}^2
				\Big) A\norm{\Pi_{C}\vY^{k}}_{C}^{2}\\
				&+\Big(1-\frac{\alpha_x r \delta}{2}
				+\frac{L^2}{B\theta_{R}\gamma_x} C c_{24}\gamma_{x}^2
				+ \frac{AL^2}{B\theta_{C}\gamma_y}  e_1 Cc_{34}\gamma_{x}^2\\
				&\quad+\frac{1}{\alpha_x r \delta}Cd_{44}\gamma_{x}^2
				+\frac{DL^2}{B\alpha_y r \delta}e_2 Cd_{54}\gamma_{x}^2
				\Big)B\norm{\tX^{k}-\vH^{k}_{x}}_{F}^{2}\\
				&+\Big(1-\frac{\alpha_y r \delta}{4}
				+\frac{A}{D\theta_{C}\gamma_y}Cc_{35}\gamma_{y}^2
				+\frac{B/L^2}{D\alpha_x r \delta}Cc_{45}L^2\hlambda^2\gamma_{y}^2\\
				&\quad+\frac{1}{\alpha_y r \delta} Cd_{55}\gamma_{y}^2
				\Big)D\norm{\vY^{k}-\vH^{k}_{y}}_{F}^{2}\\			
				&+\Big(
				\frac{1}{\theta_{R}\gamma_x}(1-\theta_{R}\gamma_x)^2c_2L^2\hlambda^2
				+\frac{A}{ \theta_{C}\gamma_y}  e_1 c_3L^2\hlambda^2\\
				&\quad+\frac{B/L^2}{\alpha_x r \delta}c_4L^2\hlambda^2
				+\frac{D}{\alpha_y r \delta}e_2 c_5L^2\hlambda^2
				\Big)\\
				&\quad~\cdot\frac{3\norm{ \vu_{C} }^{2}}{n} L^2 \norm{ \Pi_{R}\vX^{k}}_{R}^{2}\\
				&+\Big(
				\frac{1}{A\theta_{R}\gamma_x}(1-\theta_{R}\gamma_x)^2c_2L^2\hlambda^2
				+\frac{1}{ \theta_{C}\gamma_y}  e_1 c_3L^2\hlambda^2\\
				&\quad+\frac{B/L^2}{A\alpha_x r \delta}c_4L^2\hlambda^2
				+\frac{D}{A\alpha_y r \delta}e_2 c_5L^2\hlambda^2
				\Big) 3A\norm{ \Pi_{C}\vY^{k} }_{C}^{2}\\
				&+\Big(
				\frac{1}{\theta_{R}\gamma_x}(1-\theta_{R}\gamma_x)^2c_2L^2\hlambda^2
				+\frac{A}{ \theta_{C}\gamma_y}  e_1 c_3L^2\hlambda^2\\
				&\quad+\frac{B/L^2}{\alpha_x r \delta}c_4L^2\hlambda^2
				+\frac{D}{\alpha_y r \delta}e_2 c_5L^2\hlambda^2
				\Big) 3 \norm{\vu_{C}}^{2}\norm{\nabla f(\oX^{k})}^2\\
				&+s_k^{2}\zeta_0.
			\end{aligned}\label{Vkplus1}
	\end{equation}}\normalsize
	
	From the definitions of $A,B,D$, we  can bound the terms appearing in \eqref{Vkplus1} as 
	{\small
		$$\frac{A}{\theta_{C}\gamma_y}  e_1 c_{32}\gamma_{x}^2=\frac{\theta_{R}\gamma_{x}}{4},$$
		$$\frac{B/L^2}{\alpha_x r \delta}d_{42}\gamma_{x}^2=\frac{\theta_{R}\gamma_{x}}{4},$$
		$$\frac{D}{\alpha_y r \delta}e_2 d_{52}\gamma_{x}^2=\frac{\theta_{R}\gamma_{x}\theta_{C}\gamma_{y}}{4}\leq \frac{\theta_{R}\gamma_{x}}{4}.$$ }\normalsize
	Recalling the condition on $\gamma_{x},\gamma_{y},\hlambda$, 
	we get
	{\small
		$$ \frac{B/L^2}{A\alpha_x r \delta}c_{43}L^2\hlambda^2\gamma_{y}^2\leq \frac{\theta_{C}\gamma_{y}}{4},$$
		$$\frac{D}{A\alpha_y r \delta}d_{53}\gamma_{y}^2\leq \frac{\theta_{C}\gamma_{y}}{4},$$
		$$\frac{L^2}{B\theta_{R}\gamma_x} C c_{24}\gamma_{x}^2 \leq \frac{\alpha_x r \delta}{16},$$
		$$\frac{AL^2}{B\theta_{C}\gamma_y}  e_1 Cc_{34}\gamma_{x}^2\leq \frac{\alpha_x r \delta}{16},$$
		$$\frac{1}{\alpha_x r \delta}Cd_{44}\gamma_{x}^2\leq \frac{\alpha_x r \delta}{16},$$ 
		$$\frac{DL^2}{B\alpha_y r \delta}e_2 Cd_{54}\gamma_{x}^2\leq \frac{\alpha_x r \delta}{16},$$ 
		$$\frac{A}{D\theta_{C}\gamma_y}Cc_{35}\gamma_{y}^2 \leq \frac{\alpha_y r \delta}{16},$$ 
		$$\frac{B/L^2}{D\alpha_x r \delta}Cc_{45}L^2\hlambda^2\gamma_{y}^2 \leq \frac{\alpha_y r \delta}{16},$$ 
		$$\frac{1}{\alpha_y r \delta} Cd_{55}\gamma_{y}^2 \leq \frac{\alpha_y r \delta}{16},$$
		\vspace{-1.2em}
		\begin{multline*}
			\Big(
			\frac{1}{\theta_{R}\gamma_x}(1-\theta_{R}\gamma_x)^2c_2L^2\hlambda^2
			+\frac{A}{ \theta_{C}\gamma_y}  e_1 c_3L^2\hlambda^2\\
			+\frac{B/L^2}{\alpha_x r \delta}c_4L^2\hlambda^2
			+\frac{D}{\alpha_y r \delta}e_2 c_5L^2\hlambda^2
			\Big) 3E\norm{ \vu_{C} }^{2}
			\leq \frac{\theta_{R}\gamma_{x}}{32},
	\end{multline*}}\normalsize
	where $E=\frac{\norm{\vu_{R}}\norm{\vu_{C}} }{n^2M} \geq \frac{1}{n}$ holds from $M\leq\frac{1}{n}\norm{\vu_{R}}\norm{\vu_{C}}.$
	Meanwhile, we have 
	{\small
		\begin{multline*}
			3\Big(
			\frac{1}{A\theta_{R}\gamma_x}(1-\theta_{R}\gamma_x)^2c_2L^2\hlambda^2
			+\frac{1}{ \theta_{C}\gamma_y}  e_1 c_3L^2\hlambda^2\\
			+\frac{B/L^2}{A\alpha_x r \delta}c_4L^2\hlambda^2
			+\frac{D}{A\alpha_y r \delta}e_2 c_5L^2\hlambda^2
			\Big) 
			\leq \frac{\theta_{C}\gamma_{y}}{4}.
	\end{multline*}}\normalsize
	Therefore, we obtain
	{\small
		\begin{align*}
			&V^{k+1}\\
			\leq& \left( 1-\frac{7\theta_{R}\gamma_x}{32}
			\right) L^2 \Omega_{c}^{k}
			+\left( 1-\frac{\theta_{C}\gamma_y }{4} 
			\right) A\Omega_{g}^{k}\\
			&+\left(1-\frac{\alpha_x r \delta}{4}
			\right)B\Omega_{cx}^{k}
			+\left(1-\frac{\alpha_y r \delta}{16}
			\right)D\Omega_{cy}^{k}\\
			&+ \frac{\beta}{4 E}\norm{\nabla f(\oX^{k})}^2
			+s_k^{2}\zeta_0,
	\end{align*}}\normalsize
	where $\beta=\frac{\theta_{R}\gamma_{x}}{8}$ and $E=\frac{\norm{\vu_{R}}\norm{\vu_{C}} }{n^2M}$. 
	Combining Lemma \ref{Lem:descent}, we obtain
	{\small
		\begin{align*}
			&\Omega_{o}^{k+1}+\frac{E M\hlambda}{\beta} V^{k+1}
			\leq \Omega_{o}^{k}-\frac{M\hlambda}{4}   \norm{\nabla f(\oX^{k})}^2\\
			&\qquad\qquad+\frac{E M\hlambda}{\beta}\Bigg[
			\Big( 1-\frac{3\theta_{R}\gamma_x}{32}\Big) L^2 \Omega_{c}^{k}
			+\Big( 1-\frac{\theta_{C}\gamma_y }{8} \Big) A\Omega_{g}^{k}\\
			&\qquad\qquad+\Big(1-\frac{\alpha_x r \delta}{4}\Big)B\Omega_{cx}^{k}
			+\Big(1-\frac{\alpha_y r \delta}{16}
			\Big)D\Omega_{cy}^{k}
			\Bigg]
			+s_k^{2}\tilde{\zeta_0},
		\end{align*}
	}\normalsize
	where $\tilde{\zeta_0}=\zeta_0\frac{E M\hlambda}{\beta}$ and we use the fact $\gamma_{x}\leq \frac{\sqrt{M\norm{\vu_{C}}}}{\sqrt{108\norm{\vu_{R}}e_1}}\theta_{C}\gamma_y$. 
	
	\subsection{Proof of Theorem \ref{Thm:RCPP-NCVX}}\label{Pf:RCPP-NCVX}
	Based on Lemma \ref{Lem:Key}, we have
	{\small\begin{align*}
			&\Omega_{o}^{k+1}+\frac{E M\hlambda}{\beta} V^{k+1}\leq \Omega_{o}^{k}-\frac{M\hlambda}{4} \norm{\nabla f(\oX^{k})}^2\\
			&\qquad+\frac{E M\hlambda}{\beta}\Bigg[
			L^2 \Omega_{c}^{k}
			+ A\Omega_{g}^{k}
			+B\Omega_{cx}^{k}
			+D\Omega_{cy}^{k}
			\Bigg]
			+s_k^{2}\tilde{\zeta_0}\\
			&\qquad=\Omega_{o}^{k}+\frac{E M\hlambda}{\beta} V^{k}-\frac{M\hlambda}{4}   \norm{\nabla f(\oX^{k})}^2+s_k^{2}\tilde{\zeta_0}.
	\end{align*}}\normalsize
	Noticing the definition of $\Omega^{k}$, we know
	{\small
	\begin{align*}
		&\Omega^{k+1}\leq \Omega^{k}-\frac{M\hlambda}{4}   \norm{\nabla f(\oX^{k})}^2+s_k^{2}\tilde{\zeta_0}.
	\end{align*}}\normalsize
	Summing the above relation over $k=0,1,\ldots,K-1$, we obtain
	{\small
	\begin{align*}
		&\frac{M\hlambda}{4}   \sum_{k=0}^{K-1}\norm{\nabla f(\oX^{k})}^2\leq \Omega^{0}-\Omega^{K}+\frac{a_0}{1-a}\tilde{\zeta_0},
	\end{align*}}\normalsize
	where we use $\sum_{k=0}^{K-1}s_k^{2}=\sum_{k=0}^{K-1}a_{0}a^{k}=\frac{a_0(1-a^{K})}{1-a}<\frac{a_0}{1-a}$.
	The proof is thus completed.
	\subsection{Proof of Theorem \ref{Thm:RCPP-PL}}\label{Pf:RCPP-PL}
	Recalling the PL condition, we obtain $-\frac{M\hlambda}{4}\norm{\nabla f(\oX^{k})}^2\leq -\frac{M\hlambda\mu}{2} (f(\oX^{k})-f(\vX^*)) $. Based on Lemma \ref{Lem:Key}, we have
	{\small\begin{align*}
			&\Omega_{o}^{k+1}+\frac{E M\hlambda}{\beta} V^{k+1}
			\leq\tilde{\rho}(\Omega_{o}^{k} + \frac{E M\hlambda}{\beta} V^{k})+s_k^{2}\tilde{\zeta_0}\\
			\leq&\tilde{\rho}^{k+1}(\Omega_{o}^{0} + \frac{E M\hlambda}{\beta} V^{0})+\sum_{l=0}^{k}\tilde{\rho}^{k-l}c^{l}\Theta\\
			\leq&\tilde{\rho}^{k+1}(\Omega_{o}^{0} + \frac{E M\hlambda}{\beta} V^{0})+ c^{k}\Theta\sum_{l=0}^{k}\left(\frac{\tilde{\rho}}{c}\right)^{k-l}\\
			\leq&\tilde{\rho}^{k+1}(\Omega_{o}^{0} + \frac{E M\hlambda}{\beta} V^{0})+c^{k+1} \frac{\Theta}{c-\tilde{\rho}},
	\end{align*}}\normalsize
	where $\Theta=c_0\tilde{\zeta_0}$. 	Since $c\in (\tilde{\rho},1)$, we have $\tilde{\rho}^{k+1}<c^{k+1}$. It implies that the optimization error $\Omega_{o}^{k}$ and the consensus error $\Omega_{c}^{k}$ both converge to $0$ at the linear rate  $\mathcal{O}(c^{k})$.
		
		%

		%
		%

		%
		%

		\ifCLASSOPTIONcaptionsoff
		\newpage
		\fi

	\end{document}

%% file: notation.tex
\newcommand{\va}{{\mathbf{a}}}
\newcommand{\vb}{{\mathbf{b}}}

\newcommand{\vh}{{\mathbf{h}}}

\newcommand{\vu}{{\mathbf{u}}}
\newcommand{\vv}{{\mathbf{v}}}
\newcommand{\vw}{{\mathbf{w}}}
\newcommand{\vx}{{\mathbf{x}}}
\newcommand{\vy}{{\mathbf{y}}}

\newcommand{\vA}{{\mathbf{A}}}

\newcommand{\vC}{{\mathbf{C}}}

\newcommand{\vF}{{\mathbf{F}}}

\newcommand{\vH}{{\mathbf{H}}}
\newcommand{\vI}{{\mathbf{I}}}

\newcommand{\vQ}{{\mathbf{Q}}}
\newcommand{\vR}{{\mathbf{R}}}

\newcommand{\vU}{{\mathbf{U}}}
\newcommand{\vV}{{\mathbf{V}}}

\newcommand{\vX}{{\mathbf{X}}}
\newcommand{\vY}{{\mathbf{Y}}}

\newcommand{\cC}{{\mathcal{C}}}

\newcommand{\cF}{{\mathcal{F}}}

\newcommand{\cQ}{{\mathcal{Q}}}

\newcommand{\EE}{{\mathbb{E}}}
\newcommand{\RR}{\mathbb{R}}
\newcommand{\vzero}{\mathbf{0}}
\newcommand{\vone}{{\mathbf{1}}}
\newcommand{\oX}{\overline{\vX}}
\newcommand{\oY}{\overline{\vY}}
\newcommand{\ox}{\overline{\vx}}

\newcommand{\T}{\intercal}

\newcommand{\tX}{\widetilde{\vX}}
\newcommand{\tY}{\widetilde{\vY}}

\newcommand{\olambda}{\overline{\lambda}}
\newcommand{\hlambda}{\widehat{\lambda}}
\newcommand{\vzeta}{\boldsymbol{\zeta}}

\newcommand{\norm}[1]{\left\| #1 \right\|}      